%% file: main.tex
\documentclass[onefignum,onetabnum]{siamart190516}

\input{shared}

\ifpdf
\hypersetup{
  pdftitle={A Robust Two-level additive Schwarz Preconditioner For Sparse Normal Equations},
  pdfauthor={H. Al Daas, P. Jolivet, and J.~A. Scott}
}
\fi

\newcommand{\Inte}[1]{\mathcal{R}_{#1}}
\newcommand{\spn}[1]{\text{span}\left\{#1\right\}} % span
\begin{document}

\maketitle

% REQUIRED
\begin{abstract}
Solving the normal equations corresponding to large sparse linear
least-squares problems is an important and challenging
problem. For very large problems, an iterative
solver is needed and, in general, a preconditioner is required
to achieve good convergence.
In recent years, a number of preconditioners have been proposed.
These are largely serial and reported results demonstrate that
none of the commonly used preconditioners for the normal equations
matrix is capable of solving all sparse least-squares problems.
Our interest is thus in 
designing new preconditioners for the normal equations that are efficient, robust,
and can be implemented in parallel. Our proposed preconditioners can be constructed
efficiently and algebraically without any knowledge of the problem and without
any assumption on the least-squares matrix except that it is sparse.
We exploit the structure of the symmetric positive definite normal equations matrix
and use the concept of algebraic local symmetric positive semi-definite splittings
to introduce  two-level Schwarz preconditioners for least-squares problems.
The condition number of the preconditioned normal equations is shown to be theoretically bounded
independently of the number of subdomains in the splitting. This upper bound can be adjusted using
a single parameter $\tau$ that the user can specify.
We discuss how the new preconditioners can be implemented on top of the PETSc library
using only 150 lines of Fortran, C, or Python code.
Problems arising from practical applications are used to compare the performance of the 
proposed new preconditioner with that of other preconditioners.
\end{abstract}

% REQUIRED
\begin{keywords}
  Algebraic domain decomposition, two-level preconditioner, additive Schwarz, normal equations, 
  sparse linear least-squares.
\end{keywords}

% REQUIRED
%\begin{AMS}
%  68Q25, 68R10, 68U05
%\end{AMS}

\section{Introduction}
We are interested in solving large-scale {sparse} linear  least-squares~(LS) problems
\begin{equation}
  \label{eq:Ax-b}
  \min_x\|Ax-b\|_2,
\end{equation}
where $A \in \R^{m\times n}$  ($m \ge n$) and $b\in\R^m$ are given.
Solving \cref{eq:Ax-b} is mathematically equivalent to solving the  $n \times n$  normal equations
   \begin{equation} \label{eq:normal}
         C x = A^\top b, \qquad C = A^\top A,
   \end{equation}
   where, provided $A$ has full column rank, the normal equations matrix $C$ is symmetric and positive definite (SPD).
Two main classes of methods may be used to solve the normal equations:
direct methods and iterative methods. A direct method proceeds by computing an explicit factorization,
either using a sparse Cholesky factorization of~$C$  or a ``thin'' QR factorization of
$A$.
While well-engineered direct solvers \cite{AguBGL16,Dav11,MKL} are highly robust, iterative methods may be preferred
because they generally  require significantly less storage (allowing them to
tackle very large problems for which the memory requirements of a direct solver
are prohibitive) and, in some applications, it may
not be necessary to solve the system with the high accuracy offered by a direct solver.
However, the successful application of an iterative method usually requires a suitable
preconditioner to achieve acceptable (and ideally, fast) convergence rates. Currently,
there is much less knowledge of preconditioners for LS problems than there
is for sparse symmetric linear systems and, as observed in Bru et al. \cite{BruMMT14}, ``the
problem of robust and efficient iterative solution of LS problems is much
harder than the iterative solution of systems of linear equations.'' This is, at least in
part, because $A$ does not have the properties of differential problems that can make
standard preconditioners effective for solving many classes of linear systems.

Compared with other classes of linear systems,
the development of preconditioners for sparse LS problems may be regarded as still 
being in its infancy. {Approaches include}
\begin{itemize}
    \item variants of block Jacobi (also known as block
      Cimmino) and SOR \cite{Elf80};
    \item incomplete factorizations such as incomplete Cholesky, QR, and LU factorizations,
    for example, \cite{BruMMT14,lisa:06,sctu:2016,sctu:2017b};
    \item sparse approximate inverses \cite{CuiH09}.
\end{itemize}
A review and performance comparison is given in \cite{GouS17}.
This found that, whilst none of the approaches {successfully solved}
all LS problems, limited memory incomplete Cholesky factorization preconditioners 
appear to be the most reliable.
The incomplete factorization-based preconditioners are designed for moderate size
problems because current approaches, in general, are not suitable for parallel computers. 
The block Cimmino method can be parallelized easily, however, it lacks
robustness as the iteration count to reach convergence cannot be controlled and
typically increases significantly when the number of blocks increases for a fixed
problem \cite{DufGRZ15}. Several techniques have been proposed to improve the
convergence of block Cimmino but they still lack robustness \cite{DumLPRT18}.
Thus, we are motivated to design a new class of
LS preconditioners that are not only reliable but can also be implemented in parallel.

{We restrict our study in this paper to the case where $C$ is sparse.
We observe that in some practical applications
the matrix $A$ contains a small number of rows that have many more
nonzero entries than the other rows, resulting in a dense matrix $C$. Several techniques, including matrix stretching and using the augmented system,
have been proposed to handle this type of problem. These result in solving a transformed system of sparse normal equations, see for example \cite{ScoT21}
and the references therein.
}

In \cite{AldG19}, Al Daas and Grigori presented a class of robust 
fully algebraic two-level additive Schwarz preconditioners
for  solving SPD linear systems of equations. They introduced the notion of an algebraic local symmetric
positive semi-definite~(SPSD) splitting of an SPD matrix  with respect to local subdomains. They used this
splitting to construct a class of second-level spaces that bound
the spectral condition number of the preconditioned system by a user-defined value. Unfortunately, Al Daas and Grigori
reported that for general sparse SPD matrices, constructing the splitting is prohibitively expensive.
Our interest is in examining whether the particular structure
of the normal equations matrix allows the approach to be successfully used for preconditioning LS problems.
%In this paper, we show how it is not necessary to explicitly compute the normal equations matrix  
%but that, given $A$, the splitting can be computed implicitly. \cred{But when we implement it we do compute
%$C$ explicitly so we may want to reword here?}
In this paper, we show how to compute the splitting efficiently.
Based on this splitting, we apply the theory presented in \cite{AldG19}
to construct a two-level Schwarz preconditioner for the normal equations.

Note that for most existing preconditioners of the normal equations, there
is no need to form and store  the normal equations matrix  $C$ explicitly. 
For example, the lower triangular part of
its columns can be computed one at a time, used to perform the corresponding step of
an incomplete Cholesky algorithm, and then discarded. However, forming the normal
equations matrix, even piecemeal, can entail a significant overhead and can potentially  lead to a severe loss
of information in highly ill-conditioned cases.
Although building our proposed preconditioner does not need the explicit 
computation of $C$, our parallel implementation computes 
it efficiently and uses it to setup the preconditioner. This is mainly motivated by technical
reasons. As an example, state-of-the-art distributed-memory graph partitioners
such as ParMETIS~\cite{KarK98} or PT-SCOTCH~\cite{PelF96}
cannot directly  partition the columns of the \emph{rectangular} matrix $A$.
Our numerical experiments on highly ill-conditioned LS problems showed that forming $C$ 
and using a positive diagonal shift to construct
the preconditioner had no major effect on the robustness of the resulting preconditioner.

This paper is organized as follows.
The notation used in the manuscript is given at the end of the introduction.
In \cref{sec:dd}, we present an overview of domain decomposition (DD) methods for a sparse
SPD matrix.
%as interested readers
%may not be familiar with this family of methods. This also helps make the manuscript self
%contained. Sufficient and necessary theoretical results required to establish
%the robustness of the proposed preconditioner are also recalled.
We present a framework for the DD approach when applied to the sparse LS
problem in~\cref{sec:normal}. Afterwards, we show how to compute the local SPSD
splitting matrices efficiently and use them in line with the theory presented in
\cite{AldG19} to construct a robust two-level Schwarz preconditioner for the normal
equations matrix.
We then discuss some technical details that clarify how to construct the 
preconditioner efficiently.
In \cref{sec:numerical_experiments}, we briefly discuss how the new preconditioner 
can be implemented on top of the PETSc library \cite{PETSc} and
we illustrate its effectiveness using large-scale
LS problems coming from practical applications.
Finally, concluding comments are made in \cref{sec:conclusion}.

\paragraph{Notation}
We end our introduction by defining notation that will be used in this paper.
Let $1 \le n \le m$ and let 
$A \in \mathbb{R}^{m \times n}$. 
Let $S_1 \subset \llbracket 1, m \rrbracket$ and $S_2 \subset \llbracket 1, n\rrbracket$ be two sets of integers.
$A(S_1, :)$ is the submatrix of $A$ formed by the rows whose indices belong to
$S_1$ and $A(:, S_2)$ is the submatrix of $A$ formed by the columns whose indices belong to $S_2$.
The matrix $A(S_1,S_2)$ is formed by taking the rows whose indices belong
to~$S_1$ and only retaining the columns whose indices belong to $S_2$.
The concatenation
of any two sets of integers $S_1$ and $S_2$ is represented by $[S_1, S_2]$. Note that the order of the concatenation is
important. The set of the first $p$ positive integers is denoted by $\llbracket 1,p\rrbracket$. 
The identity matrix of size $n$ is denoted by $I_n$.
We denote by $ker(A)$ and $range(A)$  the null space and the range of $A$, respectively.

\section{Introduction to domain decomposition}
\label{sec:dd}
Throughout this section, we assume that $C$ is a general $n \times n$ sparse SPD matrix.
Let the nodes $V$ in the corresponding adjacency graph ${\cal G}(C)$ be numbered from $1$ to $n$. 
A graph partitioning algorithm can be used to split $V$ into $N \ll n$ disjoint subsets $\part{Ii}$ ($1\le i \le N$)
of size $\nomega{Ii}$. These sets are called  nonoverlapping subdomains.
Defining an {overlapping} additive Schwarz preconditioner requires overlapping subdomains.
Let $\part{\Gamma i}$ be the subset of size  $\nomega{\Gamma i}$ of  nodes that are distance one in ${\cal G}(C)$ from 
the nodes in $\part{Ii}$  ($1\le i \le N$).
The overlapping subdomain $\part{i}$ is defined to be $\part{i}=[\part{Ii}, \part{\Gamma i}]$,
with size $\nomega{i} = \nomega{\Gamma i} + \nomega{Ii}$.

Associated with $\Omega_{i}$ is a  restriction (or projection) matrix 
$R_{i}\in \mathbb{R}^{n_i \times n}$ given by
$R_i = I_n(\part{i},:)$.
$R_i$ maps from the global domain to subdomain $\part{i}$. Its transpose
$R_{i}^\top$ is a prolongation matrix that maps from subdomain $\part{i}$ to the global domain.
The {\em one-level additive Schwarz preconditioner} \cite{DolJN15} is defined to be
\begin{equation}
  \label{eq:one_level_schwarz}
  \schwarz{ASM} = \sum_{i=1}^N \rest{i} C_{ii}^{-1} \res{i},  \hspace{0.5cm} C_{ii} = \res{i} C \rest{i}.
\end{equation}
That is,
\begin{displaymath}
  \schwarz{ASM} = \Int{1}
  \begin{pmatrix}
    C_{11}^{-1}\\ & \ddots \\ & & C_{NN}^{-1}
  \end{pmatrix}
  \Int{1}^\top, 
\end{displaymath}
where $\Int{1}$ is the one-level interpolation operator defined by
\begin{align*}
		\begin{split}
      \Int{1} \ \colon \	&	\prod_{i = 1}^N \mathbb{R}^{\nomega{i}}	 \to 				\mathbb{R}^{n}							\\
													&	\left( u_i \right)_{ 1 \leq i \leq N}	 \mapsto 		\sum_{i = 1}^N \rest{i} u_i.
		\end{split}
\end{align*}
Applying this preconditioner to a vector involves solving concurrent local
problems in the overlapping subdomains.
Increasing $N$ reduces the sizes $n_i$ of the overlapping subdomains, 
leading to smaller local problems and faster computations.
However, in practice, the preconditioned system using $\schwarz{ASM}$ may not be well-conditioned, inhibiting convergence
of the iterative solver. In fact, the local nature of this preconditioner can lead to a deterioration 
in its effectiveness as the number of subdomains increases because of the lack of global 
information from the matrix~$C$~\cite{DolJN15,GanL17}.
To maintain robustness with respect to $N$, 
an artificial subdomain is added 
to the preconditioner (also known as second-level correction or coarse correction) that includes global information.

Let $0 < n_0 \ll n$. If $\res{0} \in \mathbb{R}^{n_0\times n}$ is of full row rank, 
the {\em two-level additive Schwarz preconditioner} \cite{DolJN15} is defined to be 
\begin{equation}
  \label{eq:two_level_schwarz}
  \schwarz{additive} = \sum_{i=0}^N \rest{i} C_{ii}^{-1} \res{i} = \rest{0} C_{00}^{-1} \res{0} + \schwarz{ASM}, \hspace{0.5cm} C_{00} = \res{0} C \rest{0}.
\end{equation}
That is,
\begin{displaymath}
  \schwarz{additive} = \Int{2}
  \begin{pmatrix}
    C_{00}^{-1}\\ & C_{11}^{-1}\\ & &  \ddots \\ & & & C_{NN}^{-1}
  \end{pmatrix}
  \Int{2}^\top, 
\end{displaymath}
where $\Int{2}$ is the two-level interpolation operator
\begin{align}
				\label{eq:interpolation2}
				\begin{split}
          \Int{2} \ \colon \	\prod_{i = 0}^N \mathbb{R}^{\nomega{i}}	& \to						\mathbb{R}^{n}							\\
													\left( u_i \right)_{ 0 \leq i \leq N}			& \mapsto 			\sum_{i = 0}^N \rest{i} u_i.
				\end{split}
\end{align}
In the rest of this paper, we will make use of the canonical 
one-to-one correspondence between $\prod_{i = 0}^N \mathbb{R}^{\nomega{i}}$ 
and $\mathbb{R}^{\sum_{i = 0}^N \nomega{i}}$ so that $\Int{2}$ can be applied to vectors in $\mathbb{R}^{\sum_{i = 0}^N \nomega{i}}$.
Observe that, because $C$ and $\res{0}$ are of full rank, $C_{00}$ is also of full rank.
For any full rank $R_0$, it is possible
to cheaply obtain upper bounds on the largest eigenvalue 
of the preconditioned matrix, independently of $n$ and $N$ \cite{AldG19}.
However, bounding the smallest eigenvalue is highly dependent on  $R_0$.
Thus, the choice of $R_0$ is key to 
obtaining a well-conditioned system and building efficient two-level Schwarz preconditioners.
Two-level Schwarz preconditioners have been used to solve a large class of systems arising from
a range of engineering applications (see, for example,
\cite{HeiHK20,JolRZ21,KonC17,MarCJNT20,SmiBG96,SpiDHNPS14,VanSG09} and references therein).

%Note that the dimensions of $\schwarz{1}$ and $\schwarz{2}$ are $\sum_{i = 1}^N\nomega{i} > N$ 
%and $\sum_{i = 0}^N\nomega{i} > N$, respectively.

Following \cite{AldG19}, we denote by $D_i \in \mathbb{R}^{n_i\times n_i}$ ($1\le i \le N$) 
any non-negative diagonal matrices such that
\begin{equation} \label{eq:partition unity}
  \sum_{i=1}^N \rest{i} D_{i} \res{i} = I_n.
\end{equation}
We refer to $\left(D_i\right)_{1 \le i \le N}$ as an \emph{algebraic partition of unity}.
%Note that $D_i$ is not uniquely defined.
In \cite{AldG19}, Al Daas and Grigori show how to select local subspaces $\Z{i} \in \R^{n_i \times p_i}$ 
with $p_i \ll n_i$ ($1 \le i \le N$)   
such that, if $\rest{0}$ is defined to be $\rest{0} = [\rest{1}D_1\Z{1},  \ldots,  \rest{N}D_N\Z{N}]$,
the spectral condition number of the preconditioned matrix $\schwarz{additive} C$ is bounded from above
independently of $N$ and $n$.

\subsection{Algebraic local SPSD splitting of an SPD matrix}
We now recall the definition of an algebraic local SPSD splitting of an SPD matrix
given in \cite{AldG19}. This requires some additional notation.
Denote the complement of $\part{i}$ in $\llbracket 1, n\rrbracket$ by $\part{\text{c}i}$. 
Define   restriction matrices $\res{\text{c}i}$,
$\res{Ii}$, and $\res{\Gamma i}$ that map from the global domain to $\part{\text{c}i}$, $\part{Ii}$, and $\part{\Gamma i}$, respectively.
Reordering the matrix $C$ using the permutation matrix $P_i = I_n([\part{Ii}, \part{\Gamma i}, \part{\text{c}i}], :)$ gives
the block tridiagonal matrix
\begin{equation}
  \label{eq:permutedB}
    P_i C P_i^\top = \begin{pmatrix} C_{I,i} & C_{I\Gamma,i} & \\ C_{\Gamma I,i} & C_{\Gamma, i} & C_{\Gamma \text{c}, i} \\ & C_{\text{c}\Gamma,i} & C_{\text{c},i} \end{pmatrix},
\end{equation}
where $C_{I,i} = \res{Ii} C \rest{Ii}$, $C_{\Gamma I,i}^\top = C_{I\Gamma,i} = \res{Ii} C \rest{\Gamma i}$, 
$C_{\Gamma, i} = \res{\Gamma i} C \rest{\Gamma i}$, $C_{\text{c}\Gamma,i}^\top = C_{\Gamma \text{c},i} = \res{\Gamma i} C \rest{\text{c}i}$, and $C_{\text{c},i} = \res{\text{c}i} C \rest{\text{c}i}$.
The first 
block on the diagonal corresponds to the nodes in $\part{Ii}$, the second  block on the diagonal
corresponds to the nodes in $\part{\Gamma_i}$, and the third block on the diagonal
is associated with the remaining nodes.

An {\em algebraic  local SPSD splitting} of the SPD matrix $C$ with respect to the $i$-th  subdomain is defined to be any 
SPSD matrix $\widetilde{C}_i \in \R^{n \times n}$ of the form
\begin{equation*}
  P_i \widetilde{C}_i P_i^\top = \begin{pmatrix} C_{I,i} & C_{I\Gamma,i} & 0  \\ 
  C_{\Gamma I,i} & \widetilde{C}_{\Gamma, i}  & 0 \\ 0 & 0 & 0 \end{pmatrix}
\end{equation*}
such that the following condition holds:
\begin{equation*}
  0 \le u^\top \widetilde{C}_i u \le u^\top C u, \text{\quad for all } u \in \R^n.
\end{equation*}
We denote the $2\times 2$ block nonzero matrix of $P_i \widetilde{C}_i P_i^\top$ by $\widetilde{C}_{ii}$ so that 
$$\widetilde{C}_i = \rest{i} \widetilde{C}_{ii} \res{i}.$$
Associated with the local SPSD splitting matrices, we define a multiplicity constant~$k_m$ that satisfies the inequality
\begin{equation}
  \label{eq:sum C-tildeC_ge_0}
  0 \le \sum_{i=1}^N u^\top \widetilde{C}_i u \le k_m u^\top C u, \text{\quad for all } u \in \R^n.
\end{equation}
Note that, for any set of SPSD splitting matrices, $k_m \le N$.

The main motivation for defining splitting matrices is to 
find local seminorms that are bounded from above by the $C$-norm.
These seminorms will be used to determine a subspace that contains 
the eigenvectors of $C$ associated with its smallest eigenvalues.

\subsection{Two-level Schwarz method}
\label{sec:two-level schwarz}
We next review the abstract theory of the two-level Schwarz method as presented in \cite{AldG19}.
For the sake of completeness, we present some elementary lemmas that are widely used in multilevel methods.
These will be used in proving efficiency of the two-level Schwarz preconditioner and will also help in understanding
how the preconditioner is constructed.

\subsubsection{Useful lemmas}
The following lemma  \cite{Nep91} provides a unified framework for bounding the spectral condition 
number of a preconditioned operator. It can be found in
different forms for finite and infinite dimensional spaces. Here, we follow the presentation from \cite[Lemma 7.4]{DolJN15}.
\begin{lemma}[Fictitious Subspace Lemma]
				\label{lemma:fictitious_lemma}
				Let $C \in \mathbb{R}^{n_{C} \times n_{C}}$ and $B \in \mathbb{R}^{n_{B} \times n_{B}}$ be SPD. 
        Let the operator $\Inte{}$ be  defined as
				\begin{align*}
								\begin{split}
                  \Int{} \ \colon \ \mathbb{R}^{n_B} & \to			 \mathbb{R}^{n_C} \\
                                            v								&	\mapsto  \Int{} v,
								\end{split}
				\end{align*}
				
        \noindent and let $\Inte{}^\top$ be its transpose.
				Assume the following conditions hold:
				\begin{description}
          \item(i) $\Inte{}$ is surjective;

					\item(ii) there exists $c_u > 0$ such that for all $v \in \mathbb{R}^{n_B}$ 
						\begin{equation*}
                    \left( \Int{} v \right)^\top C \left( \Int{} v \right) \leq c_u  v^\top B v;
						\end{equation*}

          \item(iii) there exists $c_l > 0$ such that  for all $ v_{C} \in \mathbb{R}^{n_C}$ 
          there exists $v_{B} \in \mathbb{R}^{n_B}$ such that $v_{C} = \Inte{} v_{B}$ and
						\begin{equation*}
                    c_l  v_{B}^\top B v_{B} \leq \left( \Inte{} v_{B} \right)^\top C \left( \Inte{} v_{B} \right) = v_{C}^\top C v_{C}.
						\end{equation*}
				\end{description}
        Then, the spectrum of the operator $ \Inte{} B^{-1} \Inte{}^\top C$ is contained in the interval $[c_l, c_u]$.
\end{lemma}
%Note that, in the finite dimensional case, it is straightforward to show the existence of $c_l$ and $c_u$. 
The challenge is to define 
the second-level projection matrix $\res{0}$ such that the two-level additive Schwarz preconditioner $\schwarz{additive}$ and the 
operator $\Int{2}$ \cref{eq:interpolation2}, corresponding respectively to $B$ and $\Int{}$ in \cref{lemma:fictitious_lemma}, 
satisfy  conditions {\em (i)} to {\em (iii)} and, in addition, ensures the ratio between $c_l$ and $c_u$ is small
because this determines the quality of the preconditioner.

As shown in \cite[Lemmas 7.10 and 7.11]{DolJN15}, a two-level additive Schwarz preconditioner 
satisfies {\em (i)} and {\em (ii)} for any full rank $R_0$.
Furthermore, the constant~$c_u$ is bounded from above independently of the number of subdomains $N$,
as shown in the following result \cite[Theorem 12]{ChaM1994}.
\begin{lemma}
				\label{lemma:ASM_upper_bound}
				Let $k_c$ be the minimum number of distinct colours 
				so that the spaces spanned by the columns of the matrices $\rest{1}, \ldots, \rest{N}$
				that are of the same colour are mutually $C$-orthogonal. 
				Then,
				\begin{equation*}
                \left( \Inte{2} u_\mathcal{B} \right)^\top C \left( \Inte{2} u_\mathcal{B} \right) \leq  
                (k_c + 1) \ \sum_{i = 0}^{N} u_i^\top C_{ii} u_i, 
				\end{equation*}
for all $ u_\mathcal{B} = \left( u_i \right)_{0 \leq i \leq N} \in \prod_{i = 0}^N \mathbb{R}^{n_{i}}$.
\end{lemma}

\medskip
Note that $k_c$  is independent of $N$. Indeed, it  depends only on the sparsity structure of $C$ and 
is less than the maximum number of neighbouring subdomains.

\medskip
The following result is the first step in a three-step approach to define a two-level additive Schwarz operator $\Inte{2}$
that satisfies  condition {\em (iii)} in \cref{lemma:fictitious_lemma}.
\begin{lemma}{\cite[Lemma 7.12]{DolJN15}}
				\label{lemma:B_upper_bounded}
        Let  $u_\mathcal{B} = \left( u_i \right)_{ 0 \leq i \leq N}  \in \prod_{i = 0}^N \mathbb{R}^{n_{i}}$ 
        and $u = \Inte{2} u_{\mathcal{B}} \in \mathbb{R}^{n}$.
        Then, provided $\res{0}$ is of full rank,
				\begin{equation*}
								\sum_{i = 0}^{N} u_i^\top C_{ii} u_i \leq  2 \ u^\top C u + \left( 2 k_c + 1 \right)  \sum_{i = 1}^N u_i^\top C_{ii} u_i,
				\end{equation*}
				where $k_c$ is defined in \cref{lemma:ASM_upper_bound}.
\end{lemma}

It follows that {\em (iii)} is satisfied if the squared 
localized seminorm $u_i^\top C_{ii} u_i$  is bounded from above by the squared $C$-norm of $u$.

\medskip
In the second step, we  bound $u_i^\top C_{ii} u_i$ 
by the squared localized seminorm defined by the SPSD splitting matrix $\widetilde{C}_i$, 
which can be bounded by the squared $C$-norm \cref{eq:sum C-tildeC_ge_0}.
  The decomposition of $u = \sum_{i=0}^N \rest{i} u_i \in \R^n$ is termed {\em stable} if, for some $\tau > 0$,
  \begin{equation*}
    \tau u_i^\top C_{ii} u_i \le u^\top C u, \qquad 1 \le i \le N.
  \end{equation*}
  The two-level approach in \cite{AldG19} aims to decompose each 
$\R^{\nomega{i}}$ ($1 \le i \le N$) into two subspaces, 
  one that makes the decomposition of $u$ stable and the other is part of the artificial subdomain
  associated with the second level of the preconditioner.
  Given the partition of unity \cref{eq:partition unity}, $u = \sum_{i=1}^N \rest{i} D_i \res{i} u$ and, if
  $\Pi_i = \Pi_i^\top\in \R^{\nomega{i} \times \nomega{i}}$ , we can write 
  \begin{align*}
    u &= \sum_{i=1}^N \rest{i} D_i (I_{\nomega{i}} - \Pi_i) \res{i} u + \sum_{i=1}^N \rest{i} D_i \Pi_i \res{i} u\\
    %u &= \sum_{i=1}^N \underbrace{\rest{i} D_i (I_{\nomega{i}} - \Pi_i) \res{i} u}_{\in \text{ stable subdomain}} + \underbrace{\sum_{i=1}^N \rest{i} D_i \Pi_i \res{i} u}_{\in \text{ artificial subdomain}},\\
      &= \sum_{i=1}^N \rest{i} u_i + \sum_{i=1}^N \rest{i} D_i \Pi_i \res{i} u, 
      \quad \mbox{with } u_i = D_i \left( I_{n_{i}} - {\Pi}_{i} \right) \res{i} u.
  \end{align*}
  Therefore, we need to construct $\Pi_i$ such that
  \begin{equation*}
    \tau u^\top \rest{i} (I_{\nomega{i}} - \Pi_i) D_i C_{ii} D_i (I_{\nomega{i}} - \Pi_i) \res{i} u \le u^\top C u.
  \end{equation*}
The following lemma shows how this can be done.
%\begin{lemma}
%				\label{lemma:Filtration1}
%				Let $B, \, C \in \mathbb{R}^{m \times m}$ be two symmetric positive semi-definite matrices.
%        Let $ker(B)$, $range(B)$ denote the null space and the range of $B$ respectively. Let $ker(C)$ denote the kernel of $C$.
%				Let $L = ker(B) \cap ker(C)$, we note $L^{\perp_{ker(B)}}$ the orthogonal complementary of $L$ in $ker(B)$.
%				Let $P_0$ be an orthogonal projection on $range(B)$.
%				Let $\tau$ be a strictly positive real number.
%				Consider the generalized eigenvalue problem,
%				\begin{align}
%					\label{eq:Filtration_GEVP1}
%					\begin{split}
%									&P_0 C P_0 u_k  =  \lambda_k B u_k,	\\
%									&u_k \in range(B),\\
%									&\lambda_k \in \mathbb{R}.
%					\end{split}
%				\end{align}
%				Let $P_{\tau}$ be an orthogonal projection on the subspace 
%				\begin{displaymath}
%						Z = L^{\perp_{ker(B)}} \, \oplus \, \spn{ u_k\, | \, \lambda_k > \tau },
%				\end{displaymath}
%
%				\noindent then, the following inequality holds:
%				\begin{equation}
%								\label{eq:deflated_local_operator1}
%								\left( u - P_{\tau}  u \right)^\top C \left( u - P_{\tau} u \right) \leq \tau u^\top B u, \ \forall u \in \mathbb{R}^m.
%				\end{equation}
%				Furthermore, $Z$ is the subspace of smallest dimension such that \cref{eq:deflated_local_operator1} holds.
%\end{lemma}
\begin{lemma}{\cite[Lemma 4.2]{AldG19}}
				\label{lemma:local_filtering_subspace}
        Let $\widetilde{C}_i = \rest{i}\widetilde{C}_{ii}\res{i}$ be a local SPSD splitting of $C$ 
        related to the \mbox{$i$-th} subdomain ($1 \le i \le N$). Let $D_i$ be the partition of unity \cref{eq:partition unity}.
        Let ${P}_{0,i}$ be the projection on $range(\widetilde{C}_{ii})$ parallel to $ker (\widetilde{C}_{ii})$.
        Define $L_i = ker( D_i C_{ii} D_i) \cap ker(\widetilde{C}_{ii})$, and let $L_i^{\perp}$ denote the orthogonal complementary of 
        $L_i$ in $ker(\widetilde{C}_{ii})$. 
				Consider the following generalized eigenvalue problem:
				\begin{align*}
								& {\text find} \ ( v_{i,k},\lambda_{i,k} ) \ \in \mathbb{R}^{n_{i}} \times \mathbb{R} \\
                & {\text such \ that } \ P_{0,i} D_i C_{ii} D_i P_{0,i} v_{i, k} = \lambda_{i,k} \widetilde{C}_{ii} v_{i, k}.
				\end{align*}
				Given $\tau > 0$, define
				\begin{equation}
								\label{eq:local_eigenspace2}
                \mathcal{Z}_{i} = L_i^{\perp} \oplus \spn{ v_{i, k} \ | \ \lambda_{i, k} > \dfrac{1}{\tau} }
				\end{equation}
								and let  ${\Pi}_{i}$ be the orthogonal projection on $\mathcal{Z}_{i}$.
				Then, $\mathcal{Z}_{i}$ is the subspace of smallest dimension  such that for all $u \in \R^n$,
				\begin{displaymath}
          \tau u_i^\top C_{ii} u_i \leq u^\top \widetilde{C}_{i} u \le u^\top C u, 
				\end{displaymath}
				where $u_i = D_i \left( I_{n_{i}} - {\Pi}_{i} \right) \res{i} u$. 
\end{lemma}

\medskip
\Cref{lemma:stable_decomposition} provides the last step that we need for condition {\em (iii)} in \cref{lemma:fictitious_lemma}.
It defines $u_0$ and checks whether $(u_i)_{0\le i \le N}$ is a stable decomposition.

\begin{lemma}
				\label{lemma:stable_decomposition}
				Let $ \widetilde{C}_i$,  $\mathcal{Z}_{i}$, and ${\Pi}_{i}$ be as in \cref{lemma:local_filtering_subspace}
				and let $Z_i$ be a matrix whose columns span $\mathcal{Z}_i$
        ($1 \le i \le N$).
        Let the columns of the matrix $\rest{0}$ span the space 
        				\begin{equation}
                 \label{eq:Z space}
          \mathcal{Z} = \bigoplus_{i = 1}^N \rest{i} D_i {Z}_i.
          \end{equation}
				Let $u \in \mathbb{R}^n$ and 
				$u_i = D_i \left( I_{n_{i}} - {\Pi}_{i} \right) \res{i} u $ ($1 \le i \le N$).
				Define
				\begin{displaymath}
								u_0 = \left( \res{0} \rest{0} \right)^{-1} \res{0} \left( \sum_{i = 1}^N \rest{i} D_i {\Pi}_{i} \res{i} u \right).
				\end{displaymath}
				Then,
				\begin{displaymath}
								u = \sum_{i = 0}^N \rest{i} u_i,
				\end{displaymath}
				and
				\begin{displaymath}
          \sum_{i = 0}^N u_i^\top C_{ii} u_i \leq  \left( 2 + (2 k_c + 1) \dfrac{k_m}{\tau} \right) u^\top C u.
				\end{displaymath}

\end{lemma}

\medskip
Finally, using the preceding results, \cref{th:condition_number_bounded} presents a
theoretical upper bound on the spectral condition number of the preconditioned system.
\begin{theorem}
				\label{th:condition_number_bounded}
        If the two-level additive Schwarz preconditioner $\schwarz{additive}$ \cref{eq:two_level_schwarz} is constructed 
        using $\res{0}$ as defined in \cref{lemma:stable_decomposition}, then the following inequality is satisfied:
				\begin{displaymath}
          \kappa \left( \schwarz{additive} C \right) \leq (k_c + 1) \left( 2 + (2 k_c + 1) \dfrac{k_m}{\tau} \right).
				\end{displaymath}
\end{theorem}

\subsection{Variants of the Schwarz preconditioner}
So far, we have presented~$\schwarz{ASM}$, the symmetric additive Schwarz method (ASM) 
and $\schwarz{additive}$, the additive correction for the second level.
It was noted in \cite{CaiS99} that using the partition of unity 
to weight the preconditioner can improve its quality.
The resulting preconditioner is referred to as $\schwarz{RAS}$, the {\em
restricted additive Schwarz}~(RAS) preconditioner,
and is defined to be 
\begin{equation}
  \label{eq:RAS}
  \schwarz{RAS} = \sum_{i=1}^N \rest{i} D_i C_{ii}^{-1}\res{i}.
\end{equation}
This preconditioner is nonsymmetric and thus can only be used with iterative
methods such as GMRES~\cite{SaaS86}
that are for solving  nonsymmetric problems.
With regards to the second level, different strategies  yield either a
symmetric or a nonsymmetric preconditioner \cite{TanNVE09}.
Given a first-level preconditioner $\schwarz{$\star$}$ 
and setting $Q = \rest{0} C_{00}^{-1} \res{0}$, the balanced and deflated  
two-level preconditioners are as follows
\begin{equation}
 \label{eq:balanced}
  \schwarz{balanced} = Q + (I - CQ )^\top \schwarz{$\star$} (I - CQ ),
\end{equation}
and
\begin{equation}
  \label{eq:deflated}
  \schwarz{deflated} = Q + \schwarz{$\star$} (I - CQ ),
\end{equation}
respectively.
It is well-known in the literature that
$\schwarz{balanced}$ and $\schwarz{deflated}$ yield better convergence behavior
than $\schwarz{additive}$ (see~\cite{TanNVE09} for a thorough comparison).
Although the theory we present relies on $\schwarz{additive}$, in practice we will
use $\schwarz{balanced}$ and $\schwarz{deflated}$. 
If the one-level preconditioner $\schwarz{$\star$}$
is symmetric, then so is $\schwarz{balanced}$, while $\schwarz{deflated}$ is
typically nonsymmetric. 
For this reason, in the rest of the paper, we always
couple~$\schwarz{ASM}$ with  $\schwarz{balanced}$, and $\schwarz{RAS}$ with
$\schwarz{deflated}$. All three
variants have the same setup cost, and only differ in how the second
level is applied. $\schwarz{balanced}$ is slightly more expensive because
two second-level corrections (multiplications by $Q$) are required 
instead of a single one for $\schwarz{additive}$ and $\schwarz{deflated}$.

\section{The normal equations}
\label{sec:normal}
The theory explained thus far is fully algebraic but somewhat
disconnected from our initial LS problem~\cref{eq:Ax-b}.
We now show how it can be readily applied 
to the normal equations matrix $C = A^\top A$, with $A\in \R^{m \times n}$ sparse,
first  defining a one-level Schwarz preconditioner, and then a robust algebraic second-level correction.
We start by partitioning the $n$ columns of $A$ into disjoint subsets $\part{Ii}$.
Let $\rows{i}$ be the set of indices of the nonzero rows in $A(:, \part{Ii})$  and let
$\rows{\text{c}i}$ be the complement of $\rows{i}$ in the set $\llbracket 1, m\rrbracket$. 
Now define $\part{\Gamma i}$ to be the complement of $\part{Ii}$ in the set of indices of nonzero columns of $A(\rows{i},:)$. 
The set $\part{i} = [\part{Ii}, \part{\Gamma i}]$ defines the $i$-th overlapping subdomain
and we have the permuted matrix
\begin{equation}
  \label{eq:reorderedA}
    A([\rows{i}, \rows{\text{c}i}], [\part{Ii}, \part{\Gamma i}, \part{\text{c}i}]) = 
    \begin{pmatrix}A_{I,i} & A_{I\Gamma,i} & \\ & A_{\Gamma,i} & A_{\text{c},i} \end{pmatrix}.
\end{equation}
To illustrate the concepts and notation, consider the $5 \times 4$ matrix  
\begin{equation*}
  A = \begin{pmatrix}
    1 & 0 & 6 & 0\\  
    2 & 4 & 0 & 0\\
    3 & 0 & 0 & 0\\
    0 & 5 & 0 & 7\\
    0 & 0 & 0 & 8
 \end{pmatrix}
\end{equation*}
 and set $N=2$, $\Omega_{I1} = \{1, 3\}$, $\Omega_{I2} = \{2,4\}$.
Consider the first subdomain. We have
\begin{equation*}
  A(:, \Omega_{I1}) = \begin{pmatrix}
    1 & 6\\  
    2 & 0\\
    3 & 0\\
    0 & 0\\
    0 & 0
 \end{pmatrix}.
\end{equation*}
The set of indices of the nonzero rows is
$\rows{1}     = \{1, 2, 3\}$, 
and its complement is
$\rows{\text{c}1} = \{4, 5\}$.
To define $\Omega_{\Gamma,1}$,  select the nonzero columns in the submatrix $A(\rows{1}, :  )$ and 
remove those already in $\Omega_{I1}$, that is,
\begin{equation}\label{eq:rows1}
  A(\rows{1}, :  ) = \begin{pmatrix}
    1 & 0 & 6 & 0\\  
    2 & 4 & 0 & 0\\
    3 & 0 & 0 & 0
 \end{pmatrix},
\end{equation}
so that $\Omega_{\Gamma1} = \{2\}$ and $\Omega_{\text{c}1} = \{4\}$.
Permuting $A$ to the form  \cref{eq:reorderedA} gives
\begin{equation*}
    A([\rows{1}, \rows{\text{c}1}], [\Omega_{I1}, \Omega_{\Gamma1}, \Omega_{\text{c}1}] )  = \begin{pmatrix}
    1 & 6 & 0 & 0\\  
    2 & 0 & 4 & 0\\
    3 & 0 & 0 & 0\\
    0 & 0 & 5 & 7\\
    0 & 0 & 0 & 8
 \end{pmatrix}.
\end{equation*}
In the same way, consider  the second subdomain.
$\Omega_{I2} = \{2,4\}$ and
\begin{equation*}
  A(:, \Omega_{I2}) = \begin{pmatrix}
    0 & 0\\  
    4 & 0\\
    0 & 0\\
    5 & 7\\
    0 & 8
 \end{pmatrix},
\end{equation*}
so that $\rows{2} = \{2, 4, 5\}$ and $\rows{\text{c}2} = \{1, 3\}$.
To define $\Omega_{\Gamma2}$, select the nonzero columns in the submatrix $A(\rows{2}, :  )$ 
and remove those already in $\Omega_{I2}$, that is,
\begin{equation}\label{eq:rows2}
  A(\rows{2}, :  ) = \begin{pmatrix}
    2 & 4 & 0 & 0\\
    0 & 5 & 0 & 7\\
    0 & 0 & 0 & 8
 \end{pmatrix},
\end{equation}
which gives $\Omega_{\Gamma2} = \{1\}$ and  $\Omega_{\text{c}2} = \{3\}$.
Permuting $A$ to the form  \cref{eq:reorderedA} gives
\begin{equation*}
    A([\rows{2}, \rows{\text{c}2}], [\Omega_{I2}, \Omega_{\Gamma2}, \Omega_{\text{c}2}] ) 
  = \begin{pmatrix}
     4 &    0  &   2 &    0\\
     5 &    7  &   0 &    0\\
     0 &    8  &   0 &    0\\
     0 &    0  &   1 &    6\\
     0 &    0  &   3 &    0
 \end{pmatrix}.
\end{equation*}
Now that we have $\Omega_{Ii}$ and $\Omega_{\Gamma i}$, we can define the restriction operators 
\begin{equation*}
  \res{1} = I_4(\Omega_1,:) = \begin{pmatrix} 1 & 0 & 0 & 0 \\ 0 & 0 & 1 & 0 \\ 0 & 1 & 0 & 0 \end{pmatrix}, \quad 
  \res{2} = I_4(\Omega_2,:) = \begin{pmatrix} 0 & 1 & 0 & 0 \\ 0 & 0 & 0 & 1 \\ 1 & 0 & 0 & 0 \end{pmatrix}.
\end{equation*}
For our example, $n_{I1} = n_{I2}=2$ and $n_{\Gamma 1} = n_{\Gamma 2} = 1$.
The partition of unity matrices~$D_i$ are of dimension $(n_{Ii} + n_{\Gamma i})
\times (n_{Ii} + n_{\Gamma i})$ ($i = 1,2$) and have ones on the $n_{Ii}$
leading diagonal entries and zeros elsewhere, so that
\begin{equation}
  \label{eq:partition_of_unity}
  D_1 = D_2 = \begin{pmatrix} 1 & 0 & 0 \\ 0 & 1 & 0 \\ 0 & 0 & 0 \end{pmatrix}.
\end{equation}
Observe that $D_i(k,k)$ scales the columns $A(:, \Omega_i(k))$.

Note that it is possible to obtain the partitioning sets and the 
sets of indices using the normal equations matrix $C$.
Most graph partitioners, especially those that are implemented in parallel,
require an undirected graph (corresponding to a square matrix with a symmetric sparsity pattern). Therefore, in practice,
we use the graph of $C$ to setup the first-level preconditioner for LS problems.

\subsection{One-level DD for the normal equations}
This section presents the one-level additive Schwarz preconditioner for the normal equations matrix $C= A^\top A$.
Following \cref{eq:one_level_schwarz} and 
given the sets $\Omega_{Ii}, \Omega_{\Gamma i},$ and $\rows{i}$, the 
one-level Schwarz preconditioner of $C = A^\top A$ is
\begin{align*}
  \begin{split}
    \schwarz{ASM} &= \sum_{i=1}^N \rest{i}\left(\res{i} A^\top A\rest{i} \right)^{-1} \res{i},\\
                  &= \sum_{i=1}^N \rest{i}\left( A(:,\Omega_i)^\top A(:,\Omega_i) \right)^{-1} \res{i},\\
  \end{split}
\end{align*}
\begin{remark}
  \label{remark:no_explicit_form_1}
Note that the local matrix $C_{ii} = A(:,\Omega_i)^\top A(:,\Omega_i)$ need not be computed explicitly to be factored.
Instead, the Cholesky factor of $C_{ii}$ can be computed by using a ``thin'' QR factorization of $A(:,\Omega_i)$.
\end{remark}

\subsection{Algebraic local SPSD splitting of the normal equations matrix}
In this section, we show how to cheaply construct algebraic local SPSD splittings for sparse 
matrices of the form $C=A^\top A$.
Combining \cref{eq:permutedB} and \cref{eq:reorderedA}, we can write
\begin{equation*}
    P_i A^\top A P_i^\top  = \begin{pmatrix} A_{I,i}^\top A_{I,i} & A_{I,i}^\top A_{I\Gamma,i} & \\ A_{I\Gamma,i}^\top A_{I,i} & A_{I\Gamma, i}^\top A_{I\Gamma, i} + A_{\Gamma,i}^\top A_{\Gamma,i} & A_{\Gamma,i}^\top A_{\text{c}, i} \\ & A_{\text{c},i}^\top A_{\Gamma,i} & A_{\text{c},i}^\top A_{\text{c},i} \end{pmatrix},
\end{equation*}
where $P_i = I_n([\Omega_{Ii},\Omega_{\Gamma i}, \Omega_{\text{c}i}], :)$ is a permutation matrix.
A straightforward splitting of $P_i A^\top A P_i^\top$ is given by
\begin{align*}
  \begin{split}
    P_i A^\top A P_i^\top  &= \begin{pmatrix} A_{I,i}^\top A_{I,i} & A_{I,i}^\top A_{I\Gamma,i} & 0 \\ 
    A_{I\Gamma,i}^\top A_{I,i} & A_{I\Gamma, i}^\top A_{I\Gamma, i} & 0 \\ 0 & 0 & 0 \end{pmatrix}
      +  \begin{pmatrix} 0 & 0 &0 \\ 0 & A_{\Gamma,i}^\top A_{\Gamma,i} & A_{\Gamma,i}^\top A_{\text{c}, i} \\ 
      0 & A_{\text{c},i}^\top A_{\Gamma,i} & A_{\text{c},i}^\top A_{\text{c},i} \end{pmatrix}.\\
    %&= \begin{pmatrix} A_{I,i} & A_{I\Gamma,i} & 0  \end{pmatrix}^\top \begin{pmatrix} A_{I,i} & A_{I\Gamma,i} & 0 \end{pmatrix} + \begin{pmatrix} 0 & A_{I,i} & A_{I\Gamma,i} \end{pmatrix}^\top \begin{pmatrix} 0 & A_{I,i} & A_{I\Gamma,i} \end{pmatrix}.
  \end{split}
\end{align*}
It is clear that both summands are SPSD.
Indeed, they both have the form $X^\top X$, where $X$ is 
$\begin{pmatrix} A_{I,i} & A_{I\Gamma,i} & 0\end{pmatrix}$ and $\begin{pmatrix} 0 & A_{\Gamma,i} & A_{\text{c},i} \end{pmatrix}$, respectively. 
The local SPSD splitting matrix related to the $i$-th subdomain is then defined as:
\begin{align}
  \label{eq:spsd_normal_eq}
  \begin{split}
    \widetilde{C}_{ii} &= A(\rows{i}, \Omega_i)^\top A(\rows{i}, \Omega_i) = \begin{pmatrix} A_{I,i} & A_{I\Gamma,i} \end{pmatrix}^\top \begin{pmatrix} A_{I,i} & A_{I\Gamma,i} \end{pmatrix},\\
  \end{split}
\end{align}
and 
$$\widetilde{C}_{i} = \rest{i} \widetilde{C}_{ii} \res{i} = A(\rows{i}, :)^\top A(\rows{i}, :).$$
Hence, the theory presented in \cite{AldG19} 
and summarised in \cref{sec:two-level schwarz} is applicable.
In particular, the two-level Schwarz preconditioner $\schwarz{additive}$ \cref{eq:two_level_schwarz} satisfies 
%=======
%\cred{The following is a direct copy of what is in \Cref{sec:two-level schwarz} so
%we don't need it here? Is there anything new?}
%  Let $P_{0,i}$ be a projection on $range(\widetilde{C}_{ii})$. 
%  Let $L = ker(\widetilde{C}_{ii}) \cap ker(D_i C_{ii} D_i)$ 
%  and define $L^\perp$ to be the complementary of $L$ in $ker(\widetilde{C}_{ii})$. 
%  Let $\tau > 0$ and define the generalized eigenvalue problem:
%\begin{equation}
%  \label{eq:GEVP}
%  P_{0,i} D_i C_{ii} D_i P_{0,i} v = \lambda \widetilde{C}_i v.
%\end{equation}
%Define $\mathcal{Z}_i = L^\perp \oplus span\{v_k | \lambda_k > \tau\}$, 
%where $(v_k, \lambda_k)$ satisfies \cref{eq:GEVP}, and let $Z_i$ be a matrix whose columns span $\mathcal{Z}_i$.
%
%Define the subspace  $\mathcal{Z}$ as:
%\begin{equation}
%  \label{eq:define_coarse_space}
%  \mathcal{Z} = \bigoplus_{i = 1}^N \rest{i} D_i \mathcal{Z}_i.
%\end{equation}
%and set $\res{0}$ to be a matrix whose columns span the subspace $\mathcal{Z}$ in $\R^n$.
%\cred{in  Lemma 2.5 you have $\rest{0}$ spans the subspace $\mathcal{Z}$.
%Is it $\res{0}$ or $\rest{0}$?}
%Then, the two-level Schwarz preconditioner $\schwarz{2}$ using $\res{0}$ as a coarse 
%space restriction matrix satisfies 
%>>>>>>> 1b6543d9e839f6d14e816e8f28318cd88d4b0d9a
$$\kappa(\schwarz{additive} C) \le (k_c + 1) \left(2 + 2(k_c+1)\dfrac{k_m}{\tau}\right), $$
where $k_c$ is the minimal number of colours required to colour the partitions of $C$ such that 
each two neighbouring subdomains have different colours, and $k_m$ is the multiplicity constant that satisfies the following inequality
\begin{equation*}
  \sum_{i=1}^N \rest{i} \widetilde{C}_{ii} \res{i} \le k_m C.
\end{equation*}
The constant $k_c$ is independent of $N$ and depends only on the graph $\mathcal{G}(C)$, which is determined by the sparsity pattern of $A$. 
The multiplicity constant $k_m$ depends on the local SPSD splitting matrices.
For the normal equations matrix, the following lemma provides an upper bound on $k_m$.
\begin{lemma}
  \label{lemma:k_m}
    Let $C = A^\top A$. Let $m_j$ be the number of subdomains such that $A(j, \Omega_{Ii}) \neq 0$ ($1\leq i \leq N$),
  that is, 
\begin{equation*}
  m_j = \#\{i \ | \ j \in \rows{i}\}.
\end{equation*}
  Then, $k_m$ can be chosen to be $k_m = \max_{1\le j \le m} m_j$.
    Furthermore, if $k_{\Omega_i}$ is the number of neighbouring subdomains of the \mbox{$i$-th} subdomain,
  that is,
  $$k_{\Omega_i} = \#\{j \ | \ \Omega_i \cap \Omega_j \neq \phi \},$$ 
  then
  \begin{equation*}
   \label{eq:k_m_lemma}
  k_m = \max_{1\le j \le m} m_j \le \max_{1 \le i \le N} k_{\Omega_i}.
  \end{equation*}
\end{lemma}
\begin{proof}
  Since $C = A^\top A$ and $\widetilde{C}_i = A(\rows{i}, :)^\top A(\rows{i}, :)$,  
  we have
  \begin{align*}
    u^\top C u &= \sum_{j=1}^m u^\top A(j, :)^\top A(j, :) u,\\
    u^\top \widetilde{C}_i u &= \sum_{j \in \rows{i}} u^\top A(j, :)^\top A(j, :) u,\\
    \sum_{i = 1}^N  u^\top \widetilde{C}_i u &= \sum_{i = 1}^N \sum_{j \in \rows{i}} u^\top A(j, :)^\top A(j, :) u.
  \end{align*}
  From the definition of $m_j$, the term $u^\top A(j, :)^\top A(j, :) u$ appears $m_j$ times in the last equation. Thus,
  \begin{align*}
    \sum_{i = 1}^N  u^\top \widetilde{C}_i u &= \sum_{j = 1}^m m_j u^\top A(j, :)^\top A(j, :) u,\\
                                             &\le \max_{1\le j \le m} m_j  \sum_{j = 1}^m u^\top A(j, :)^\top A(j, :) u,\\
                                             &= \max_{1\le j \le m} m_j (u^\top C u),
  \end{align*}
from which it follows that we can choose $k_m = \max_{1\le j \le m} m_j$.
Now, if $1 \le l \le m$, there exist $i_1, \ldots, i_{m_l}$ such that $l \in \rows{i_1} \cap \cdots \cap \rows{i_{m_l}}$.
Furthermore, $m_l \le \max_{1 \le p \le l} k_{\Omega_{i_p}}$.
%<<<<<<< HEAD
% Taking the maximum over $l$ on both sides we obtain~\cref{eq:k_m_lemma}.
%=======
Taking the maximum over $l$ on both sides, we obtain
  \begin{equation} \tag*{~}
    k_m \le \max_{1 \le i \le N} k_{\Omega_{i}}.
  \end{equation}
%>>>>>>> cea3680b5caca1eea48952cbcefa3a312e16e423
\end{proof}
Note that because $A$ is sparse, $k_m$ is independent of the number of subdomains.
\subsection{Algorithms and technical details}
%\begin{algorithm}
%  \caption{Sparse matrix parser}
%  \label{alg:2lvl}
%  \begin{algorithmic}[1]
%    \Require{Sparse matrix $B$}
%    \Ensure{$\rows$ nonzero rows of $B$, $\Omega$ nonzero columns of $A$}
%    \Function{$[\rows{},\Omega]= $Find}{$(B)$}
%    \State{}
%  \end{algorithmic}
%\end{algorithm}
In this section, we discuss the technical details involved in constructing a two-level preconditioner 
for the normal equations matrix.

\subsubsection{Partition of unity}
\label{sec:partition of unity}
  Because the matrix $A_{I \Gamma,i}$ may be of low rank, the null space of $\widetilde{C}_{ii}$ 
  \cref{eq:spsd_normal_eq} can be large.
  Recall that the diagonal matrices $D_i$ have dimension $n_i = n_{Ii}+n_{\Gamma i}$.
  Choosing the  entries in positions $n_{Ii}+1, \ldots, n_{i}$ 
  of the diagonal of $D_i$ to be zero, as in \cref{eq:partition_of_unity},
  results in the subspace of $ker(\widetilde{C}_{ii})$ caused by the rank deficiency of $A_{I \Gamma,i}$ to lie 
  within $ker(D_iC_{ii}D_i)$, reducing the size of the space $\mathcal{Z}$ given by \cref{eq:Z space}.
  In other words, if $A_{I \Gamma,i} u  = 0$, we have $ \widetilde{C}_{ii}v = 0$, 
  where $v^\top = (0, u^\top)$, i.e., $v\in ker(\widetilde{C}_{ii})$ and because
  by construction $D_i v = 0$, we have $v \in ker(\widetilde{C}_{ii}) \cap ker(D_i C_{ii} D_i)$, therefore, 
  $v$ need not be included in $\mathcal{Z}_i$.
\subsubsection{The eigenvalue problem}
The generalized eigenvalue problem presented in \cref{lemma:local_filtering_subspace} 
is critical in the construction of the two-level preconditioner.
Although the definition of $\mathcal{Z}_i$ from \cref{eq:local_eigenspace2} suggests 
it is necessary to compute the null space of $\widetilde{C}_{ii}$ and that of $D_i C_{ii} D_i$ and their intersection,
in practice, this can be avoided.
Consider the generalized eigenvalue problem
\begin{equation}
  \label{eq:gevp}
  D_i C_{ii} D_i v = \lambda \widetilde{C}_{ii} v,
\end{equation}
where, by convention, we set $\lambda = 0$ if 
$v \in ker(\widetilde{C}_{ii}) \cap ker(D_i C_{ii} D_i)$ and $\lambda = \infty$ if $v \in ker(\widetilde{C}_{ii}) \setminus ker(D_i C_{ii} D_i)$.
The subspace $\mathcal{Z}_i$ defined in \cref{eq:local_eigenspace2} can then be written as
\begin{equation*}
  \spn{v \ | \ D_i C_{ii} D_i v = \lambda \widetilde{C}_{ii} v \text{ and } \lambda > \dfrac{1}{\tau}}.
\end{equation*}
Consider also the shifted generalized eigenvalue problem
\begin{equation}
  \label{eq:modified_gevp}
  D_i C_{ii} D_i v = \lambda (\widetilde{C}_{ii} + s  I_{n_i}) v,
\end{equation}
where $0 < s \ll 1$.
Note that if $s$ is such that $\widetilde{C}_{ii} + s  I_{n_i}$ is numerically of full rank, 
\cref{eq:modified_gevp} can be solved using any off-the-shelf generalized eigenproblem solver.
Let $(v,\lambda)$ be an eigenpair of \cref{eq:modified_gevp}. Then, we can only have one of the following situations:
\begin{itemize}
  \item $v \in range(\widetilde{C}_{ii}) \cap ker(D_i C_{ii} D_i)$ or $v \in ker(\widetilde{C}_{ii}) \cap ker(D_i C_{ii} D_i)$.
    In which case, $(v, 0)$ is an eigenpair of \cref{eq:gevp}.

  \item $v \in range(\widetilde{C}_{ii}) \cap range(D_i C_{ii} D_i)$. Then,
\[
  \frac{\| D_i C_{ii} D_i v - \lambda \widetilde{C}_{ii} v \|_2}{\lambda \| v\|_2} = s,
\]
and, as $s$ is small, $(v, \lambda)$ is a good approximation of an eigenpair of \cref{eq:gevp} corresponding to a finite eigenvalue.

  \item $v \in ker(\widetilde{C}_{ii}) \cap range(D_i C_{ii} D_i)$. Then,
$D_i C_{ii} D_i v = \lambda s v$, i.e., $\lambda s$ is a nonzero eigenvalue of $D_i C_{ii} D_i$.
Because $D_i$ is defined such that the diagonal values corresponding to the boundary nodes are zero, 
the nonzero eigenvalues of $D_i C_{ii} D_i$ correspond to the squared singular values of $A(:, \Omega_{Ii})$.
Hence, all the eigenpairs of \cref{eq:gevp} corresponding to an infinite eigenvalue are included 
in the set of eigenpairs $(v,\lambda)$ of \cref{eq:modified_gevp} such that  
\begin{equation}
  \label{eq:inequality lambda s}
  \sigma^2_{\text{min}}\left(A(:, \Omega_{Ii})\right) \le \lambda s \le \sigma^2_{\text{max}}\left(A(:, \Omega_{Ii})\right),
\end{equation}
where $\sigma_{\text{min}}\left(A(:, \Omega_{Ii})\right)$ 
and $\sigma_{\text{max}}\left(A(:, \Omega_{Ii})\right)$ are the smallest and largest singular values of $A(:, \Omega_{Ii})$, respectively.
\end{itemize}
\medskip \noindent
Therefore, choosing $$s = O(\|\widetilde{C}_{ii}\|_2 \varepsilon),$$ where $\varepsilon$ 
is the machine precision, ensures $\widetilde{C}_{ii} + s  I_{n_i}$ is numerically invertible and $s \ll 1$.
Setting $s = \|\widetilde{C}_{ii}\|_2 \varepsilon$ in \cref{eq:inequality lambda s}, we obtain
\[
  \sigma^2_{\text{min}}\left(A(:, \Omega_{Ii})\right) \le \lambda \|\widetilde{C}_{ii}\|_2 \varepsilon \le \sigma^2_{\text{max}}\left(A(:, \Omega_{Ii})\right).
\]
By \cref{eq:spsd_normal_eq}, we have $$\|\widetilde{C}_{ii}\|_2 \le \| C_{ii} \|_2,$$ 
and 
because $\Omega_{Ii}\subset \Omega_i$, it follows that
$$\|C_{ii}^{-1}\|_2 = \|\left(A(:, \Omega_{i})^\top A(:, \Omega_{i}) \right)^{-1}\|_2 \le \sigma^2_{\text{min}}\left(A(:, \Omega_{Ii})\right).$$
Hence, if $(v,\lambda)$ is an eigenpair of \cref{eq:modified_gevp} with $v \in ker(\widetilde{C}_{ii}) \cap range(D_i C_{ii} D_i)$,
then
\[
  (\kappa(C_{ii}) \varepsilon)^{-1} \le \lambda,
\]
where $\kappa(C_{ii})$ is the condition number of $C_{ii}$
and $\mathcal{Z}_i$ can be defined to be
\begin{equation}
  \label{eq:in practice Zi}
  %\spn{v \ | \ D_i C_{ii} D_i v = \lambda (\widetilde{C}_{ii} + s  I_{n_i}) v \text{ and } \lambda \ge \min\left(\tau, s^{-1} \sigma^2_{\text{min}}\left(A(:, \Omega_{Ii})\right)\right)},
  \spn{v \ | \ D_i C_{ii} D_i v = \lambda (\widetilde{C}_{ii} + \varepsilon \|\widetilde{C}_{ii}\|_2  I_{n_i}) v 
  \text{ and } \lambda \ge \min\left(\dfrac{1}{\tau}, (\kappa(C_{ii}) \varepsilon)^{-1}\right)}.
\end{equation}
$Z_i$ is then taken to be the matrix whose columns are the vertical concatenation of corresponding eigenvectors.

\begin{remark}
  \label{remark:no_explicit_form_2}
Note that solving the generalized eigenvalue problem \cref{eq:modified_gevp} by an iterative method such as 
Krylov--Schur \cite{Ste02} does not require the explicit form of $C_{ii}$ and~$\widetilde{C}_{ii}$. Rather, it
requires solving linear systems of the form $(\widetilde{C}_{ii} + s I_{n_i}) u = v$,
together with matrix--vector products
of the form $(\widetilde{C}_{ii} + s I_{n_i}) v$ and $C_{ii} v$.
It is clear that these products do not require 
the matrices $\widetilde{C}_{ii}$ and $C_{ii}$ to be formed.
Regarding the solution of the linear system $(\widetilde{C}_{ii} + s I_{n_i}) u = v$, 
\cref{remark:no_explicit_form_1} also applies to 
the Cholesky factorization of $\widetilde{C}_{ii} + s I_{n_i} = X^\top X$, 
where $X^\top = \begin{pmatrix} A(\rows{i}, \Omega_i)^\top  & \sqrt{s} I_{n_i}\end{pmatrix}$,
that can be computed by using a ``thin'' QR factorization of~$X$.
\end{remark}

From \cref{remark:no_explicit_form_1,remark:no_explicit_form_2}, and applying the same technique
therein to factor $C_{00} = \res{0} C \rest{0} = (A\rest{0})^\top (A\rest{0})$, we observe that 
given the overlapping partitions of $A$, the proposed two-level preconditioner can be constructed
without forming the normal equations matrix. Algorithm~\ref{alg:2lvl} gives an overview of the 
steps for constructing our two-level Schwarz preconditioner for the normal equations matrix. The
actual implementation of our proposed preconditioner will be discussed in greater detail 
in~\cref{sec:implementation}.
\begin{algorithm}
  \caption{Two-level Schwarz preconditioner  for the normal equations matrix.}
  \label{alg:2lvl}
  \begin{algorithmic}[1]
    \Require{matrix $A$, number of subdomains $N$, threshold $\tau$ to bound the condition number.}
    \Ensure{two-level preconditioner $\schwarz{}$ for $C=A^\top A$.}
      \State{$(\Omega_{I1}, \ldots, \Omega_{IN})=\text{Partition}(A,N)$}
    \For{$i=1$ to $N$ in parallel}
      \State{$\rows{i} = \text{FindNonzeroRows}(A(:, \Omega_{Ii}))$}
    \State{$\Omega_i = [\Omega_{Ii},\Omega_{\Gamma i}] = \text{FindNonzeroColumns}(A(\rows{i}, :))$}
    \State{Define $D_i$  as in \cref{sec:partition of unity} and $R_i$ as in \cref{sec:dd}}
    \State{Perform Cholesky factorization of $C_{ii} = A(:, \Omega_i)^\top A(:, \Omega_i)$, see~\cref{remark:no_explicit_form_1}}
    \State{Perform Cholesky factorization of $\widetilde{C}_{ii} = A(\rows{i}, \Omega_i)^\top A(\rows{i}, \Omega_i)$, 
    possibly using a small shift $s$,  see~\cref{remark:no_explicit_form_2}}
    \State{Compute $Z_i$ as defined in \cref{eq:in practice Zi}}
    \EndFor
    \State{Set $\rest{0} = \left[\rest{1}D_1 Z_1, \ldots, \rest{N}D_N Z_N\right]$}
    \State{Perform Cholesky factorization of  $C_{00} = (A\rest{0})^\top (A\rest{0})$}
    \State{Set $\schwarz{} = \schwarz{additive} = \sum_{i=0}^N \rest{i} C_{ii}^{-1} \res{i}$ or $\schwarz{balanced}$~\cref{eq:balanced} or $\schwarz{deflated}$~\cref{eq:deflated}}
  \end{algorithmic}
\end{algorithm}

\section{Numerical experiments}
\label{sec:numerical_experiments}
In this section, we illustrate the effectiveness of the new two-level
LS preconditioners $\schwarz{balanced}$ and $\schwarz{deflated}$, their robustness
with respect to the number of subdomains, 
and their efficiency in tackling large-scale sparse and ill-conditioned 
LS problems selected from the SuiteSparse Matrix Collection \cite{DavH11}. 
The test matrices are listed in \cref{tab:data_set}. 
For each matrix, we report its dimensions, the number of entries in $A$ and in 
the normal equations matrix $C$, and the condition number of $C$ (estimated using the MATLAB function {\tt condest}).

\pgfplotstableread{table.dat}\loadedtable
\begin{table}
  \caption{Test matrices taken from the SuiteSparse Matrix Collection.}
  \centering
  \label{tab:data_set}
\pgfplotstabletypeset[every head row/.style={before row=\hline,after row=\hline},
                      every last row/.style={after row=\hline},
                      every even row/.style={before row={\rowcolor[gray]{0.9}}},
                      columns={name,m,n,A-nnz,C-nnz,condest},sort,sort key=n,%,SPQR
                      columns/m/.style={int detect,column name=$m$,dec sep align},
                      columns/n/.style={int detect,column name=$n$,dec sep align},
                      columns/A-nnz/.style={int detect,column name=nnz($A$),dec sep align},
                      columns/C-nnz/.style={int detect,column name=nnz($C$),dec sep align},
                      columns/condest/.style={column name=condest($C$),precision=1,sci zerofill},
                      columns/SPQR/.style={string type,column name=SPQR,string replace={KO}{\ding{55}},string replace={OK}{\ding{51}}},
                      columns/name/.style={string type},display columns/0/.style={column name=Identifier, column type={l|}}]\loadedtable
\end{table}

In \cref{sec:implementation}, we  discuss  our implementation based on the
parallel backend~\cite{PETSc}. 
In particular, we show that very little coding effort is needed to 
construct all the necessary algebraic tools, and that it is possible to 
take advantage of
an existing package, such as HPDDM~\cite{JolRZ21}, to setup the
new preconditioners efficiently.
We then show in \cref{sec:comparison} how $\schwarz{balanced}$  and
$\schwarz{deflated}$ perform compared to 
other preconditioners when solving challenging LS problems.
The preconditioners we consider are:
\begin{itemize}
  \item limited memory incomplete Cholesky (IC) factorization specialized for 
  the normal equations matrix as implemented in {\tt HSL\_MI35} from the HSL library~\cite{HSL}
  (note that this package is written in Fortran and we run it using the supplied MATLAB interface with default 
  parameter settings);
  \item one-level overlapping Schwarz methods $\schwarz{ASM}$ and $\schwarz{RAS}$ as implemented in PETSc;
  \item algebraic multigrid methods as implemented both in BoomerAMG from the HYPRE library~\cite{FalY02} and in GAMG~\cite{AdaBKP04} from PETSc.
\end{itemize}
%We observe that the MATLAB code {\tt ichol} is frequently used as an IC factorization
%preconditioner but, in general, its performance is not competitive with that of 
%the HSL IC codes (see, for example, the 
%results presented in \cite{AldRS21}) and so it is not included here.
Finally, in \cref{sec:scalability}, we study  the strong scalability of $\schwarz{balanced}$
and its robustness with respect to the number of subdomains by using a fixed problem and increasing the number of subdomains.
% Finally, we present numerical experiments to assess the theoretical results obtained for the upper bound on the condition number and the ability to reduce it using the only one parameter $\tau$ \cref{eq:gevp}.

With the exception of the serial IC code {\tt HSL\_MI35}, all the numerical experiments are performed on
Irène, a system composed of \pgfmathprintnumber[assume math mode=true]{2292}
nodes with two 64-core AMD Rome processors clocked at
\SI{2.6}{\giga\hertz} and, 
unless stated otherwise, 256 MPI processes are used. For the domain decomposition methods, one subdomain is assigned per process. {All computations are performed
in double-precision arithmetic.}

In all our experiments, the vector $b$ in~\cref{eq:Ax-b} is generated randomly and the initial guess 
for the iterative solver is zero.
When constructing our new two-level 
preconditioners, with the exception of the results presented in \cref{fig:convergence}, 
at most 300 eigenpairs are computed on each subdomain  and 
the threshold parameter $\tau$ from~\cref{eq:in practice Zi} is set to~0.6. 
These parameters were found to provide good numerical performance after a very
quick trial-and-error approach on a single problem.  We did not want to adjust
them for each problem from~\cref{tab:data_set}, but it will be shown next that
they are fine overall without additional tuning.

\subsection{Implementation aspects}\label{sec:implementation}
The new two-level preconditioners are implemented on top of the well-known distributed memory library PETSc. 
This section is not aimed at PETSc specialists. Rather, we want 
to briefly explain what was needed to provide an efficient yet concise implementation. 
Our new code is open-source, available at \url{https://github.com/prj-/aldaas2021robust}.
It comprises fewer than 150 lines of code (including the initialization and error analysis). 
The main source files, written in Fortran, C, and Python, have three major phases, which we now outline.

\subsubsection{Loading and partitioning phase}
First, PETSc is used to load the matrix $A$ in parallel, following a contiguous one-dimensional 
row partitioning among MPI processes. We explicitly assemble  the normal equations matrix using 
the routine MatTransposeMatMult~\cite{MccSZ15}. The initial PETSc-enforced parallel 
decomposition of $A$ among processes may not be appropriate for the normal equations, 
so ParMETIS is used by PETSc to repartition $C$. 
This also induces a permutation of the columns of~$A$.

\subsubsection{Setup phase}
To ensure that the normal equations matrix $C$ is definite
{and its Cholesky factorization is breakdown free},  $C$ is shifted by 
$10^{-10} \|C\|_F I_n$ (here and elsewhere, $\|\!\cdot\!\|_F$ denotes the 
Frobenius norm). 
Note that this is only needed for the construction of the preconditioner; the preconditioner
is used to solve the original LS problem.
Given the indices of the columns owned by a MPI process, 
we call the routine MatIncreaseOverlap on the normal equations matrix to build an extended set of column indices 
of $A$ that will be used to define overlapping subdomains. 
These are the $\Omega_i$ as defined in~\cref{eq:reorderedA}. Using the  routine MatFindNonzeroRows,
this extended set of  indices  
is used to  concurrently find on each subdomain the set of nonzero rows. 
These are the sets $\rows{i}$ as illustrated in \cref{eq:rows1,eq:rows2}. The subdomain matrices $C_{ii}$ 
from \cref{eq:one_level_schwarz} as well as the partition of unity $D_i$ as illustrated in \cref{eq:partition_of_unity}
are automatically assembled by PETSc when using domain decomposition preconditioners such as PCASM or PCHPDDM. 
The right-hand side matrices of the generalized eigenvalue problems~\cref{eq:gevp} are assembled  
using MatTransposeMatMult, but note that this product is this 
time performed concurrently on each subdomain.
The small shift~$s$ from~\cref{eq:modified_gevp} is set to $10^{-8}\|\widetilde{C}_{ii}\|_F$.
These matrices and the sets of overlapping column indices are passed to PCHPDDM using routine PCHPDDMSetAuxiliaryMat. 
The rest of the setup is hidden from the user. It includes solving the generalized eigenvalue problems 
using SLEPc~\cite{HerRV05}, followed by the assembly and redistribution of the second-level 
operator using a Galerkin product \cref{eq:two_level_schwarz} (see~\cite{HPDDM} for more details on 
how this is performed efficiently in PCHPDDM).

\subsubsection{Solution phase}
For the solution phase, users can choose between multiple Krylov methods, including 
LSQR~\cite{PaiS82} and GMRES. {We use left-preconditioned LSQR (see, for example, \cite[Algorithm 2]{ArrBH14}) and right-preconditioned GMRES.} Each iteration of LSQR requires 
matrix--vector products with $A$ and $A^\top$. For GMRES, instead of 
using the previously explicitly assembled normal equations matrix, we use an implicit 
representation of the operator that computes the matrix--vector product with $A$ followed by the 
product with $A^\top$. The type of overlapping Schwarz method (additive or restricted additive) 
as well as the type of second-level correction (balanced or deflated) may be selected at runtime by the user. 
This flexibility is important because LSQR requires a symmetric preconditioner.

\subsection{Numerical validation}\label{sec:comparison}
In this section, we validate the effectiveness of the two-level method when compared to 
other preconditioners.
\Cref{tab:lsqr_comparison} presents a comparison between five preconditioners: two-level additive Schwarz
 with balanced coarse correction $\schwarz{balanced}$,
one-level additive Schwarz $\schwarz{ASM}$, BoomerAMG, GAMG, and {\tt HSL\_MI35}. 
The first level of the one- and two-level methods both use the additive Schwarz 
  formulation; the second level uses the balanced deflation 
  formulation \cref{eq:balanced}.  The results 
  are for the iterative solver LSQR. 
If $M$ denotes the preconditioner, LSQR terminates
when the LS residual satisfies 
$$\dfrac{\|\left(AM^{-1}\right)^\top(Ax-b)\|_2}{\|A\|_{M,F} \|Ax-b\|_2 } < 10^{-8},$$
where $\|A\|_{M,F} = \sum_{i=1}^n \lambda_i(M^{-1}A^\top A)$ is the sum of the positive eigenvalues of $M^{-1}A^\top A$
that is approximated by LSQR itself. 
Note that if $M^{-1}=W^{-1} W^{-\top}$, then $\|A\|_{M,F} = \|AW^{-1}\|_F$.

\begin{table}
  \caption{Preconditioner comparison when running LSQR.
    Iteration counts are reported. $\schwarz{ASM}$ and $\schwarz{balanced}$ are the one- and two-level
  overlapping Schwarz preconditioners, respectively.
    \dag~denotes iteration count exceeds \pgfmathprintnumber{1000}. 
   \ddag~denotes   either a failure in computing the preconditioner 
  because of memory issues or a breakdown of LSQR.}
  \label{tab:lsqr_comparison}
  \centering
\pgfplotstabletypeset[every head row/.style={before row=\hline,after row=\hline},
                      every last row/.style={after row=\hline},
                      every even row/.style={before row={\rowcolor[gray]{0.9}}},
                      columns={name,2L-LSQR,1L-LSQR,AMG-LSQR,GAMG-LSQR,MI35-LSQR},sort,sort key=n,
                      %columns={name,2L-LSQR,1L-LSQR,AMG-LSQR,GAMG-LSQR,MI35-LSQR,MI35-LSQR-50},sort,sort key=n,
                      display columns/1/.style={column name={$\schwarz{balanced}$},column type={c}},
                      display columns/2/.style={column name={$\schwarz{ASM}$},column type={c}},
                      columns/AMG-LSQR/.style={string type,string replace={-2}{{\ddag}}},
                      display columns/3/.style={column name=BoomerAMG,column type={c}},
                      display columns/4/.style={column name=GAMG,column type={c}},
                      columns/GAMG-LSQR/.style={string type,string replace={-1}{{\dag}},string replace={-2}{{\ddag}}},
                      columns/MI35-LSQR/.style={column name={{\tt HSL\_MI35}},column type={c},string type,string replace={-1}{{\dag}},string replace={-2}{{\ddag}}},
                      every row 0 column 5/.style={postproc cell content/.style={@cell content=\textbf{##1}}},
                      every row 1 column 1/.style={postproc cell content/.style={@cell content=\textbf{##1}}},
                      every row 2 column 1/.style={postproc cell content/.style={@cell content=\textbf{##1}}},
                      every row 3 column 1/.style={postproc cell content/.style={@cell content=\textbf{##1}}},
                      every row 4 column 1/.style={postproc cell content/.style={@cell content=\textbf{##1}}},
                      every row 5 column 1/.style={postproc cell content/.style={@cell content=\textbf{##1}}},
                      every row 6 column 1/.style={postproc cell content/.style={@cell content=\textbf{##1}}},
                      every row 7 column 1/.style={postproc cell content/.style={@cell content=\textbf{##1}}},
                      every row 8 column 3/.style={postproc cell content/.style={@cell content=\textbf{##1}}},
                      every row 8 column 5/.style={postproc cell content/.style={@cell content=\textbf{##1}}},
                      every row 9 column 3/.style={postproc cell content/.style={@cell content=\textbf{##1}}},
                      every row 10 column 3/.style={postproc cell content/.style={@cell content=\textbf{##1}}},
                      every row 11 column 5/.style={postproc cell content/.style={@cell content=\textbf{##1}}},
                      every row 12 column 1/.style={postproc cell content/.style={@cell content=\textbf{##1}}},
                      every row 13 column 1/.style={postproc cell content/.style={@cell content=\textbf{##1}}},
                      every row 14 column 1/.style={postproc cell content/.style={@cell content=\textbf{##1}}},
                      every row 15 column 5/.style={postproc cell content/.style={@cell content=\textbf{##1}}},
                      every row 16 column 1/.style={postproc cell content/.style={@cell content=\textbf{##1}}},
                      columns/name/.style={string type},display columns/0/.style={column name=Identifier,column type={l|}}]\loadedtable
\end{table}
It is clear that both the one- and two-level Schwarz methods are more robust than 
the other preconditioners as they encounter no breakdowns and solve all 
the LS problems using fewer than \pgfmathprintnumber{1000} iterations.
Because {\tt HSL\_MI35} is a sequential code that runs on a single core, there was 
not enough memory to compute the preconditioner for problem  cont11\_l. For many of the problems,
the iteration count for {\tt HSL\_MI35} can be reduced by increasing the
parameters that determine the number of entries in the IC factor (the default
values are  rather small for the large test examples).
LSQR preconditioned with BoomerAMG  breaks down for several problems, as 
reported by PETSc error code KSP\_DIVERGED\_BREAKDOWN. GAMG is  
more robust but requires more iterations for problems 
where both algebraic multigrid solvers are successful. 
Note that even with more advanced options than the default ones set by PETSc, 
such as PMIS coarsening~\cite{DeSFNY08} with extended classical 
interpolation~\cite{DeSYH06} for BoomerAMG or Schwarz smoothing for GAMG, 
these solvers do not perform considerably better numerically.
We can also see that the two-level preconditioner  outperforms 
the one-level preconditioner consistently.
% \sout{, with the exception of problem, stormG2\_1000, for which
% the normal equations matrix  is 
% not very sparse (see column 5 of~}\cref{tab:data_set})\sout{. In fact, the matrix $A$ has 121
% relatively dense rows that have more than \pgfmathprintnumber{1000} nonzeros while the rest of the rows have 
% at most 5 nonzeros per row.}

\Cref{tab:gmres_comparison} presents a similar comparison, but  using right-preconditioned GMRES  
applied directly to the normal equations~\cref{eq:normal}. A restart parameter of 100 is used. 
The relative tolerance is again set to $10^{-8}$, but this now applies to the 
unpreconditioned residual. We switch from $\schwarz{ASM}$ to $\schwarz{RAS}$ \cref{eq:RAS}, which is known to perform better numerically. 
For the two-level method, we switch from $\schwarz{balanced}$ to $\schwarz{deflated}$
 \cref{eq:deflated}.
\begin{table}
  \caption{Preconditioner comparison when running GMRES.
  Iteration counts are reported. $\schwarz{RAS}$ and $\schwarz{deflated}$ are the one- and two-level
  overlapping Schwarz preconditioners, respectively.
    \dag~denotes iteration count exceeds \pgfmathprintnumber{1000}. 
   \ddag~denotes   either a failure in computing the preconditioner 
  because of memory issues or a breakdown of GMRES.}
  \label{tab:gmres_comparison}
  \centering
\pgfplotstabletypeset[every head row/.style={before row=\hline,after row=\hline},
                      every last row/.style={after row=\hline},
                      every even row/.style={before row={\rowcolor[gray]{0.9}}},
                      columns={name,2L-GMRES,1L-GMRES,AMG-GMRES,GAMG-GMRES,MI35-GMRES},sort,sort key=n,
                      display columns/1/.style={column name={$\schwarz{deflated}$},column type={c}},
                      display columns/2/.style={column name={$\schwarz{RAS}$},column type={c}},
                      columns/AMG-GMRES/.style={column name=BoomerAMG,column type={c},string type,string replace={-1}{{\dag}},string replace={-2}{{\ddag}}},
                      columns/GAMG-GMRES/.style={column name=GAMG,column type={c},string type,string replace={-1}{{\dag}},string replace={-2}{{\ddag}}},
                      columns/1L-GMRES/.style={string type,string replace={-1}{{\dag}},string replace={-2}{{\ddag}}},
                      columns/MI35-GMRES/.style={column name={{\tt HSL\_MI35}},column type={c},string type,string replace={-1}{{\dag}},string replace={-2}{{\ddag}}},
                      every row 0 column 5/.style={postproc cell content/.style={@cell content=\textbf{##1}}},
                      every row 1 column 1/.style={postproc cell content/.style={@cell content=\textbf{##1}}},
                      every row 2 column 1/.style={postproc cell content/.style={@cell content=\textbf{##1}}},
                      every row 3 column 1/.style={postproc cell content/.style={@cell content=\textbf{##1}}},
                      every row 4 column 1/.style={postproc cell content/.style={@cell content=\textbf{##1}}},
                      every row 5 column 1/.style={postproc cell content/.style={@cell content=\textbf{##1}}},
                      every row 6 column 1/.style={postproc cell content/.style={@cell content=\textbf{##1}}},
                      every row 7 column 1/.style={postproc cell content/.style={@cell content=\textbf{##1}}},
                      every row 8 column 5/.style={postproc cell content/.style={@cell content=\textbf{##1}}},
                      every row 9 column 1/.style={postproc cell content/.style={@cell content=\textbf{##1}}},
                      every row 10 column 1/.style={postproc cell content/.style={@cell content=\textbf{##1}}},
                      every row 11 column 5/.style={postproc cell content/.style={@cell content=\textbf{##1}}},
                      every row 12 column 1/.style={postproc cell content/.style={@cell content=\textbf{##1}}},
                      every row 13 column 1/.style={postproc cell content/.style={@cell content=\textbf{##1}}},
                      every row 14 column 1/.style={postproc cell content/.style={@cell content=\textbf{##1}}},
                      every row 15 column 1/.style={postproc cell content/.style={@cell content=\textbf{##1}}},
                      every row 16 column 1/.style={postproc cell content/.style={@cell content=\textbf{##1}}},
                      columns/name/.style={string type},display columns/0/.style={column name=Identifier,column type={l|}}]\loadedtable
\end{table}
Switching from LSQR to GMRES can be beneficial for some preconditioners, 
e.g., BoomerAMG now converges in 21 iterations instead of breaking down for  
 problem mesh\_deform. But this is not always the case, 
e.g., {\tt HSL\_MI35} applied to problem deltaX does not converge
within the \pgfmathprintnumber{1000} iteration limit. The two-level 
method is the most robust approach, while the restricted additive Schwarz preconditioner  struggles to 
solve some problems, either because of a breakdown (problem stormg2-125) or 
because of  slow convergence (problems lp\_stocfor3, sgpf5y6, Hardesty2, and cont11\_l).

Recall that for the results in~\cref{tab:lsqr_comparison,tab:gmres_comparison}, the two-level 
preconditioner was constructed using at most 300 eigenpairs and the threshold parameter
$\tau$ was set to~$0.6$. 
Whilst this highlights that tuning  $\tau$ for individual problems
is not necessary to successfully solve a range of  problems,
it does not validate the ability of our preconditioner to  concurrently select
the most appropriate local eigenpairs to define an adaptive preconditioner. 
To that end, for problem watson\_2, we consider the effect on the performance of 
our two-level preconditioner of varying $\tau$. 
Results for LSQR with $\schwarz{ASM}$ and $\schwarz{balanced}$ are presented in~\cref{fig:convergence}. 
Here, 512 MPI processes are used and the convergence tolerance is again $10^{-8}$. 
We observe  that the two-level method consistently outperforms the one-level method. 
Furthermore,  as we increase $\tau$, {the iteration count
reduces} and the size $n_0$ 
of the second level increases. 
It is also interesting to highlight that the convergence is smooth even with a very small value 
$\tau = 0.01275$, $n_0 = \pgfmathprintnumber{2400}$ 
compared to the dimension  \pgfmathprintnumber{352013} of the normal equations matrix.

\pgfplotstableread{convergence.dat}\loadedconvergence
\pgfmathsetmacro{\anorm}{1.340373855230e+03}
\pgfplotstablecreatecol[create col/expr={\thisrow{N1}/(\anorm*\thisrow{R1})}]{T1}\loadedconvergence
\pgfplotstablecreatecol[create col/expr={\thisrow{N2}/(\anorm*\thisrow{R2})}]{T2}\loadedconvergence
\pgfplotstablecreatecol[create col/expr={\thisrow{N3}/(\anorm*\thisrow{R3})}]{T3}\loadedconvergence
\pgfplotstablecreatecol[create col/expr={\thisrow{N4}/(\anorm*\thisrow{R4})}]{T4}\loadedconvergence
\pgfplotstablecreatecol[create col/expr={\thisrow{N5}/(\anorm*\thisrow{R5})}]{T5}\loadedconvergence
\pgfplotstablecreatecol[create col/expr={\thisrow{N6}/(\anorm*\thisrow{R6})}]{T6}\loadedconvergence
\pgfplotstablecreatecol[create col/expr={\thisrow{N7}/(\anorm*\thisrow{R7})}]{T7}\loadedconvergence
\pgfplotstablecreatecol[create col/expr={\thisrow{N8}/(\anorm*\thisrow{R8})}]{T8}\loadedconvergence
\pgfplotstablecreatecol[create col/expr={\thisrow{N9}/(\anorm*\thisrow{R9})}]{T9}\loadedconvergence
\pgfplotstableread{size.dat}\sizetable
\begin{figure}
    \hspace*{-1cm}
    \begin{minipage}{0.71\textwidth}
\begin{tikzpicture}
\begin{semilogyaxis}[xmin=1,xmax=115,height=7cm,width=8.5cm,
    legend style={at={(0.97,0.96)},cells={anchor=west}},
    xlabel={Iteration number},
    ylabel={Relative residual},
    extra y ticks = {1e-8},
    extra y tick style = {grid=major},
    extra y tick labels={},
yticklabel={
\pgfmathparse{int(\tick)}
\ifnum\pgfmathresult=0
$1$
\else
\pgfmathparse{int(\tick/ln(10))}{$10^{\pgfmathresult}$}
\fi
},
            % load a color `cycle list' from the `colorbrewer' library
            cycle list/RdGy-6,
            % define fill color for the marker
            mark list fill={.!75!white},
            % create new `cycle list` from existing `cycle list's and an
            cycle multiindex* list={
                RdGy-6
                    \nextlist
                my marks
                    \nextlist
                [3 of]linestyles
                    \nextlist
                very thick
                    \nextlist
            },
            mark repeat=2,legend columns=1
]
\addplot+[mark repeat=9] table [x=iteration, y=T1] {\loadedconvergence};
\addplot+[mark repeat=4] table [x=iteration, y=T2] {\loadedconvergence};\label{pgfplots:tau}
\addplot+[mark repeat=3] table [x=iteration, y=T3] {\loadedconvergence};
\addplot table [x=iteration, y=T4] {\loadedconvergence};
\addplot table [x=iteration, y=T5] {\loadedconvergence};
\addplot table [x=iteration, y=T6] {\loadedconvergence};
\addplot table [x=iteration, y=T9] {\loadedconvergence};
% \addplot table [x=iteration, y=T7] {\loadedconvergence};
% \addplot table [x=iteration, y=T8] {\loadedconvergence};
    \legend{$\schwarz{ASM}$,$\tau=0.01275$,$\tau=0.02$,$\tau=0.05$,$\tau=0.1$,$\tau=0.4$,$\tau=0.6$}%,$\tau=0.9$}
%   $\tau=1.2$}
\end{semilogyaxis}
\end{tikzpicture} 
    \end{minipage}
    \begin{minipage}{0.35\textwidth}
\pgfplotstabletypeset[every head row/.style={before row=\hline,after row=\hline},
                  every last row/.style={after row=\hline},
                  every even row/.style={before row={\rowcolor[gray]{0.9}}},
                  columns/iteration/.style={column name=Iterations},
                  columns/Nc/.style={int detect,column name=$n_0$,dec sep align},
                  columns/tau/.style={column name=$\tau$,dec sep align,fixed,precision=5},
                  columns={tau,Nc,iteration},
                  font={\footnotesize},
                  ]\sizetable
    \end{minipage}
    \caption{Influence of the threshold parameter $\tau$ on the convergence of preconditioned LSQR
    for problem watson\_2 ($m=\pgfmathprintnumber[fixed]{677224}$ and $n=\pgfmathprintnumber[fixed]{352013}$).
    \label{fig:convergence}}
\end{figure}

\subsection{Performance study}\label{sec:scalability}
We next investigate the algorithmic cost of the two-level method. 
To do so, we perform a strong scaling analysis using a large problem not 
presented in~\cref{tab:data_set} but still from the SuiteSparse 
Matrix Collection, Hardesty3. The matrix is of dimension $\pgfmathprintnumber[fixed]{8217820} \times \pgfmathprintnumber[fixed]{7591564}$, 
and the number of nonzero entries in $C$ is $\pgfmathprintnumber[fixed]{98634426}$. 
In~\cref{tab:strong-scaling}, we report the number of iterations as well as the 
eigensolve, setup, and solve times as the number $N$ of subdomains ranges from 16  
to \pgfmathprintnumber{4096}. The times are obtained using the PETSc 
 -log\_view command line option. For different~$N$, 
the reported times on each row of the table are the maximum among all processes. The setup
time includes the numerical factorization of the first-level subdomain matrices, the assembly 
of the second-level operator and its factorization. Note that the symbolic factorization of 
the first-level subdomain is shared between the domain decomposition preconditioner and 
the eigensolver because we use the Krylov--Schur method as implemented in
SLEPc, which requires the factorization of the right-hand side matrices from~\cref{eq:modified_gevp}. 
The  Cholesky factorizations of the subdomain matrices and of the second-level operator are performed 
using the sparse direct solver MUMPS~\cite{AmeDLK01}. For small numbers of subdomains ($N < 128$), 
the cost of the eigensolves are clearly prohibitive.
By increasing the number of subdomains, thus 
reducing their size, the time to construct the preconditioner becomes much more tractable and overall, our implementation 
yields good speedups on a wide range of process counts.
Note that the threshold parameter $\tau=0.6$ is not attained on any of the
subdomains for $N$ ranging from 16 up to 256, so that $n_0 = 300 \times
N$. For larger $N$, $\tau=0.6$ is attained, the
preconditioner automatically selects the appropriate eigenmodes, and
convergence improves (see column 2 of~\cref{tab:strong-scaling}).
When $N$ is large ($N \geq \pgfmathprintnumber{1024}$),
the setup and solve times are impacted 
by the high cost of factorizing and solving the second-level problems, which, as highlighted by 
the values of $n_0$, become large. Multilevel variants~\cite{AldGJT19}  
could be used to overcome this but goes beyond the scope of the
current study.

\pgfplotstableread{strong.dat}\strongtable
\pgfplotstablecreatecol[create col/expr={\thisrow{setup}-\thisrow{eps}}]{rest}\strongtable
\pgfplotstablecreatecol[create col/expr={\thisrow{setup}+\thisrow{solution}}]{total}\strongtable
\pgfplotstablegetelem{0}{total}\of{\strongtable}
\let\relativeTo\pgfplotsretval
\pgfplotstablecreatecol[create col/expr={\relativeTo/\thisrow{total}}]{speedup}\strongtable
\begin{table}
    \caption{Strong scaling for problem Hardesty3 
    ($m=\pgfmathprintnumber[fixed]{8217820}$ and $n=\pgfmathprintnumber[fixed]{7591564}$) 
    for $N$ ranging from \pgfmathprintnumber{16} to \pgfmathprintnumber{4096} subdomains. All times are in seconds. 
    Column 2 reports the LSQR iteration count. Column 4 reports the setup time
    minus the concurrent solution time of the generalized eigenproblems, which
    is given in column 3. }
  \label{tab:strong-scaling}
  \centering
\pgfplotstabletypeset[every head row/.style={before row=\hline,after row=\hline},
                      every last row/.style={after row=\hline},
                      every even row/.style={before row={\rowcolor[gray]{0.9}}},
                      columns={N,iteration,eps,rest,solution,Nc,total,speedup},
                      columns/Nc/.style={int detect,column name=$n_0$,dec sep align,column type/.add={}{|},assign column name/.style={
         /pgfplots/table/column name={\multicolumn{2}{c|}{##1}}
      }},
                      display columns/7/.style={column name={Speedup},precision=1,fixed zerofill},every row 0 column 7/.style={postproc cell content/.style={@cell content={\ensuremath{-}}}},
                      display columns/6/.style={column name={Total},precision=1,fixed zerofill,dec sep align},
                      display columns/4/.style={column name={Solve},precision=1,fixed zerofill,dec sep align},
                      display columns/3/.style={column name={\hspace*{-0.1cm}Setup},precision=1,fixed zerofill,dec sep align},
                      display columns/2/.style={column name={\hspace*{-0.1cm}Eigensolve},precision=1,fixed zerofill,dec sep align},
                      display columns/1/.style={column name={Iterations}},
                      display columns/0/.style={column name=$N$,dec sep align,column type/.add={}{|},assign column name/.style={
         /pgfplots/table/column name={\multicolumn{2}{c|}{##1}}
      }}]\strongtable
\end{table}

\section{Concluding comments} \label{sec:conclusion}
%In this paper, we have proposed new two-level additive Schwarz preconditioners for the 
%normal equations matrix $C = A^\top A$.
%The approach exploits the concept of an algebraic local SPSD splitting of a SPD matrix
%that was introduced by Al Daas and Grigori \cite{AldG19}. Whereas  Al Daas and Grigori
%found that, in practice, for a general SPD matrix constructing the SPSD splitting is too expensive,
%we have shown that the structure of $C$ can be used to efficiently
%perform the splitting. Theory ...
%
%The new two-level preconditioners have been implemented within PETSc.
%Numerical experiments have demonstrated  ...
%
%Future work includes extending the approach to develop
%preconditioners for solving large sparse-dense least-squares problems in which
%$A$ contains a small number of rows that have many more entries than the other rows. These
%cause the normal equations matrix to be dense and so they need to be handled separately
%(see, for example, the recent work of Scott and T\accent23uma~\cite{sctu:2017b,sctu:2018a,sctu:2019a,sctu:2021a}).
%As already observed, we also plan to consider multilevel variants to allow the use
%of a large number of subdomains and processes.
%~ \\
Solving large-scale sparse linear least-squares problems is known to be challenging.
Previously proposed preconditioners
 have generally been serial and have involved incomplete factorizations
of $A$ or $C= A^\top A$. In this paper, we have employed ideas
that have been developed in the area of domain decomposition, which (as far as we
are aware) have not previously been applied to least-squares problems.
In particular, we have exploited recent work by Al Daas and Grigori \cite{AldG19}
on algebraic domain decomposition preconditioners
for SPD systems to propose a new two-level algebraic domain preconditioner for the
normal equations matrix $C$.
We have used the concept of an algebraic local SPSD splitting of an SPD matrix and
we have shown that the structure of $C$ as the product of $A^\top$ and $A$ can be used to efficiently
perform the splitting. Furthermore, we have proved that using the two-level preconditioner,
the spectral condition number of the preconditioned normal equations matrix is
bounded from above independently of the number of the subdomains and the size of 
the problem. Moreover, this upper bound depends on a parameter $\tau$ that can be 
chosen by the user to decrease (resp.\ increase) the upper bound with the costs of
setting up the preconditioner being larger (resp.\ smaller).

The new two-level preconditioner has been implemented in parallel within PETSc.
Numerical experiments on a range of problems from real applications have shown that
whilst both one-level and two-level domain decomposition preconditioners
are effective when used with LSQR to solve the normal equations, the latter
consistently results in significantly faster convergence.
It also outperforms other possible preconditioners, both in terms of robustness
and iteration counts. Furthermore, our numerical experiments on a set of challenging
least-squares problems show that the two-level preconditioner is robust with respect
to the parameter $\tau$. Moreover, a strong scalability test of the two-level 
preconditioner assessed its robustness with respect to the number of
subdomains.

Future work includes extending the approach to develop
preconditioners for solving large sparse--dense least-squares problems in which
$A$ contains a small number of rows that have many more entries than the other rows. These
cause the normal equations matrix to be dense and so they need to be handled separately
%(see, for example, the recent work of Scott and T\accent23uma~\cite{sctu:2017b,sctu:2018a,sctu:2019a,sctu:2021a}).
(see, for example, the recent work of Scott and T\accent23uma~\cite{ScoT21} and references therein).
As already observed, we also plan to consider multilevel variants to allow the use
of a larger number of subdomains and processes.
\appendix
\section*{Acknowledgments}
This work was granted access to the
GENCI-sponsored HPC resources of TGCC@CEA under allocation A0090607519. 
The authors would like to thank L.~Dalcin, V.~Hapla, and T.~Isaac for
their recent contributions to PETSc that made the implementation of our
preconditioner more flexible. 
{H. Al Daas and J.~A.~Scott were partially supported by EPSRC grant EP/W009676/1.
Finally, we are grateful to two anonymous
reviewers for their constructive feedback. 
}

\section*{Code reproducibility}
Interested readers are referred to \url{https://github.com/prj-/aldaas2021robust/blob/main/README.md}
for setting up the appropriate requirements, compiling, and running our proposed
preconditioner. Fortran, C, and Python source codes are provided.

\bibliographystyle{siamplain}
\bibliography{main}
\end{document}

%% file: shared.tex
\usepackage{lipsum}
\usepackage{amsfonts}
\usepackage[normalem]{ulem}
\usepackage{graphicx}
\usepackage{epstopdf}
\usepackage{pifont}
\usepackage{siunitx}
\usepackage{pgfplotstable}
    \usetikzlibrary{
        pgfplots.colorbrewer,
    }
    \pgfplotsset{
        cycle list/.define={my marks}{
            every mark/.append style={solid,fill=\pgfkeysvalueof{/pgfplots/mark list fill}},mark=*\\
            every mark/.append style={solid,fill=\pgfkeysvalueof{/pgfplots/mark list fill}},mark=square*\\
            every mark/.append style={solid,fill=\pgfkeysvalueof{/pgfplots/mark list fill}},mark=triangle*\\
            every mark/.append style={solid,fill=\pgfkeysvalueof{/pgfplots/mark list fill}},mark=diamond*\\
        },
    }
\usepackage{algpseudocode}
\usepackage{multirow}
\usepackage{adjustbox}
\usepackage{booktabs,colortbl}
\usepackage{cprotect}
\usepackage{stmaryrd}

\ifpdf
  \DeclareGraphicsExtensions{.eps,.pdf,.png,.jpg}
\else
  \DeclareGraphicsExtensions{.eps}
\fi

\newsiamremark{remark}{Remark}
\newsiamremark{hypothesis}{Hypothesis}
\crefname{hypothesis}{Hypothesis}{Hypotheses}
\newsiamthm{claim}{Claim}

\headers{Preconditioner for Sparse Normal Equations}{H. Al Daas, P. Jolivet, and J.~A. Scott}

\title{A Robust Algebraic Domain Decomposition Preconditioner for Sparse Normal Equations\thanks{Submitted to the editors \today.}}
\author{Hussam Al Daas\thanks{STFC Rutherford Appleton Laboratory, Harwell Campus, Didcot, Oxfordshire, OX11 0QX, UK 
  (\email{hussam.al-daas@stfc.ac.uk},
  \email{jennifer.scott@stfc.ac.uk}).}
  \and Pierre Jolivet\thanks{CNRS, ENSEEIHT, 2 rue Charles Camichel, 31071 Toulouse Cedex 7, France (\email{pierre.jolivet@enseeiht.fr}).}
\and Jennifer A.~Scott\footnotemark[2] \thanks{School of Mathematical, Physical and Computational Sciences,
University of Reading, Reading RG6 6AQ, UK.}
}

\usepackage{amsopn}

\newcommand{\R}{\mathbb{R}}

\newcommand{\part}[1]{\Omega_{#1}}
\newcommand{\res}[1]{R_{#1}}
\newcommand{\rest}[1]{R^\top_{#1}}
\newcommand{\Z}[1]{Z_{#1}}
\newcommand{\nomega}[1]{n_{#1}}
\newcommand{\Int}[1]{\mathcal{R}_{#1}}
\newcommand{\rows}[1]{\Xi_{#1}}
\newcommand{\schwarz}[1]{M^{-1}_{\text{\tiny{#1}}}}

%% file: main.bbl
\begin{thebibliography}{10}

\bibitem{AdaBKP04}
{\sc M.~F. Adams, H.~H. Bayraktar, T.~M. Keaveny, and P.~Papadopoulos}, {\em
  Ultrascalable implicit finite element analyses in solid mechanics with over a
  half a billion degrees of freedom}, in Proceedings of the 2004 ACM/IEEE
  Conference on Supercomputing, SC04, IEEE Computer Society, 2004,
  pp.~\mbox{34:1--34:15}.

\bibitem{AguBGL16}
{\sc E.~Agullo, A.~Buttari, A.~Guermouche, and F.~Lopez}, {\em Implementing
  multifrontal sparse solvers for multicore architectures with sequential task
  flow runtime systems}, ACM Transactions on Mathematical Software, 43 (2016),
  \url{http://buttari.perso.enseeiht.fr/qr_mumps}.

\bibitem{AldG19}
{\sc H.~Al~Daas and L.~Grigori}, {\em A class of efficient locally constructed
  preconditioners based on coarse spaces}, SIAM Journal on Matrix Analysis and
  Applications, 40 (2019), pp.~66--91.

\bibitem{AldGJT19}
{\sc H.~Al~Daas, L.~Grigori, P.~Jolivet, and P.-H. Tournier}, {\em A multilevel
  {Schwarz} preconditioner based on a hierarchy of robust coarse spaces}, SIAM
  Journal on Scientific Computing, 43 (2021), pp.~A1907--A1928.

\bibitem{AmeDLK01}
{\sc P.~R. Amestoy, I.~S. Duff, J.-Y. L'Excellent, and J.~Koster}, {\em A fully
  asynchronous multifrontal solver using distributed dynamic scheduling}, SIAM
  Journal on Matrix Analysis and Applications, 23 (2001), pp.~15--41,
  \url{http://mumps.enseeiht.fr}.

\bibitem{ArrBH14}
{\sc S.~R. Arridge, M.~M. Betcke, and L.~Harhanen}, {\em Iterated
  preconditioned {LSQR} method for inverse problems on unstructured grids},
  Inverse Problems, 30 (2014), p.~075009.

\bibitem{PETSc}
{\sc S.~Balay, S.~Abhyankar, M.~F. Adams, J.~Brown, P.~Brune, K.~Buschelman,
  L.~Dalcin, A.~Dener, V.~Eijkhout, W.~D. Gropp, D.~Karpeyev, D.~Kaushik, M.~G.
  Knepley, D.~A. May, L.~C. McInnes, R.~T. Mills, T.~Munson, K.~Rupp, P.~Sanan,
  B.~F. Smith, S.~Zampini, H.~Zhang, and H.~Zhang}, {\em {PETS}c web page},
  2021, \url{https://petsc.org}.

\bibitem{BruMMT14}
{\sc R.~Bru, J.~Mar\'{\i}n, J.~Mas, and M.~T\accent23uma}, {\em Preconditioned
  iterative methods for solving linear least-squares problems}, SIAM Journal on
  Scientific Computing, 36 (2014), pp.~A2002--A2022.

\bibitem{CaiS99}
{\sc X.-C. Cai and M.~Sarkis}, {\em A restricted additive {S}chwarz
  preconditioner for general sparse linear systems}, SIAM Journal on Scientific
  Computing, 21 (1999), pp.~792--797.

\bibitem{ChaM1994}
{\sc T.~F. Chan and T.~P. Mathew}, {\em Domain decomposition algorithms}, Acta
  Numerica, 3 (1994), pp.~61--143.

\bibitem{CuiH09}
{\sc X.~Cui and K.~Hayami}, {\em Generalized approximate inverse
  preconditioners for least-squares problems}, Japan Journal of Industrial and
  Applied Mathematics, 26 (2009).

\bibitem{Dav11}
{\sc T.~A. Davis}, {\em Algorithm 915, {SuiteSparseQR}: multifrontal
  multithreaded rank-revealing sparse {QR} factorization}, ACM Transactions on
  Mathematical Software, 38 (2011).

\bibitem{DavH11}
{\sc T.~A. Davis and Y.~Hu}, {\em The {U}niversity of {F}lorida sparse matrix
  collection}, ACM Transactions on Mathematical Software, 38 (2011), pp.~1--28.

\bibitem{DeSFNY08}
{\sc H.~De~Sterck, R.~D. Falgout, J.~W. Nolting, and U.~M. Yang}, {\em
  Distance-two interpolation for parallel algebraic multigrid}, Numerical
  Linear Algebra with Applications, 15 (2008), pp.~115--139.

\bibitem{DeSYH06}
{\sc H.~De~Sterck, U.~M. Yang, and J.~J. Heys}, {\em Reducing complexity in
  parallel algebraic multigrid preconditioners}, SIAM Journal on Matrix
  Analysis and Applications, 27 (2006), pp.~1019--1039.

\bibitem{DolJN15}
{\sc V.~Dolean, P.~Jolivet, and F.~Nataf}, {\em An introduction to domain
  decomposition methods. Algorithms, theory, and parallel implementation},
  Society for Industrial and Applied Mathematics, 2015.

\bibitem{DufGRZ15}
{\sc I.~S. Duff, R.~Guivarch, D.~Ruiz, and M.~Zenadi}, {\em The augmented block
  {Cimmino} distributed method}, SIAM Journal on Scientific Computing, 37
  (2015), pp.~A1248--A1269.

\bibitem{DumLPRT18}
{\sc A.~Dumitraşc, P.~Leleux, C.~Popa, D.~Ruiz, and S.~Torun}, {\em The
  augmented block {Cimmino} algorithm revisited}, 2018,
  \url{https://arxiv.org/abs/1805.11487}.

\bibitem{Elf80}
{\sc T.~Elfving}, {\em Block-iterative methods for consistent and inconsistent
  linear equations}, Numerische Mathematik, 35 (1980), pp.~1--12.

\bibitem{FalY02}
{\sc R.~D. Falgout and U.~M. Yang}, {\em \emph{hypre}: a library of high
  performance preconditioners}, Computational Science---ICCS 2002,  (2002),
  pp.~632--641.

\bibitem{GanL17}
{\sc M.~J. Gander and A.~Loneland}, {\em {SHEM}: an optimal coarse space for
  {RAS} and its multiscale approximation}, in Domain Decomposition Methods in
  Science and Engineering XXIII, C.-O. Lee, X.-C. Cai, D.~E. Keyes, H.~H. Kim,
  A.~Klawonn, E.-J. Park, and O.~B. Widlund, eds., Cham, 2017, Springer
  International Publishing, pp.~313--321.

\bibitem{GouS17}
{\sc N.~I.~M. Gould and J.~A. Scott}, {\em The state-of-the-art of
  preconditioners for sparse linear least-squares problems}, ACM Transactions
  on Mathematical Software, 43 (2017), pp.~36:1--35.

\bibitem{HeiHK20}
{\sc A.~Heinlein, C.~Hochmuth, and A.~Klawonn}, {\em Reduced dimension {GDSW}
  coarse spaces for monolithic {Schwarz} domain decomposition methods for
  incompressible fluid flow problems}, International Journal for Numerical
  Methods in Engineering, 121 (2020), pp.~1101--1119.

\bibitem{HerRV05}
{\sc V.~Hernandez, J.~E. Roman, and V.~Vidal}, {\em {SLEP}c: a scalable and
  flexible toolkit for the solution of eigenvalue problems}, ACM Transactions
  on Mathematical Software, 31 (2005), pp.~351--362,
  \url{https://slepc.upv.es}.

\bibitem{HSL}
{\em {HSL.} {A} collection of {Fortran} codes for large-scale scientific
  computation}, 2018.
\newblock \url{http://www.hsl.rl.ac.uk}.

\bibitem{HPDDM}
{\sc P.~Jolivet, F.~Hecht, F.~Nataf, and C.~Prud'homme}, {\em Scalable domain
  decomposition preconditioners for heterogeneous elliptic problems}, in
  Proceedings of the International Conference on High Performance Computing,
  Networking, Storage and Analysis, SC '13, New York, NY, USA, 2013, ACM,
  pp.~80:1--80:11.

\bibitem{JolRZ21}
{\sc P.~Jolivet, J.~E. Roman, and S.~Zampini}, {\em {KSPHPDDM} and {PCHPDDM}:
  extending {PETSc} with advanced {Krylov} methods and robust multilevel
  overlapping {Schwarz} preconditioners}, Computers \& Mathematics with
  Applications, 84 (2021), pp.~277--295.

\bibitem{KarK98}
{\sc G.~Karypis and V.~Kumar}, {\em Multilevel $k$-way partitioning scheme for
  irregular graphs}, Journal of Parallel and Distributed computing, 48 (1998),
  pp.~96--129.

\bibitem{KonC17}
{\sc F.~Kong and X.-C. Cai}, {\em A scalable nonlinear fluid–structure
  interaction solver based on a {Schwarz} preconditioner with isogeometric
  unstructured coarse spaces in {3D}}, Journal of Computational Physics, 340
  (2017), pp.~498--518.

\bibitem{lisa:06}
{\sc N.~Li and Y.~Saad}, {\em {MIQR}: a multilevel incomplete {QR}
  preconditioner for large sparse least-squares problems}, SIAM Journal on
  Matrix Analysis and Applications, 28 (2006), pp.~524--550.

\bibitem{MarCJNT20}
{\sc P.~Marchand, X.~Claeys, P.~Jolivet, F.~Nataf, and P.-H. Tournier}, {\em
  Two-level preconditioning for $h$-version boundary element approximation of
  hypersingular operator with {GenEO}}, Numerische Mathematik, 146 (2020),
  pp.~597--628.

\bibitem{MccSZ15}
{\sc M.~McCourt, B.~F. Smith, and H.~Zhang}, {\em Sparse matrix--matrix
  products executed through coloring}, SIAM Journal on Matrix Analysis and
  Applications, 36 (2015), pp.~90--109.

\bibitem{MKL}
{\em Intel {MKL Sparse} {QR}}, 2018.

\bibitem{Nep91}
{\sc S.~V. Nepomnyaschikh}, {\em Mesh theorems of traces, normalizations of
  function traces and their inversions}, Russian Journal of Numerical Analysis
  and Mathematical Modelling, 6 (1991), pp.~1--25.

\bibitem{PaiS82}
{\sc C.~C. Paige and M.~A. Saunders}, {\em {LSQR}: an algorithm for sparse
  linear equations and sparse least squares}, ACM Transactions on Mathematical
  Software, 8 (1982), p.~43–71.

\bibitem{PelF96}
{\sc F.~Pellegrini and J.~Roman}, {\em {SCOTCH}: a software package for static
  mapping by dual recursive bipartitioning of process and architecture graphs},
  in High-Performance Computing and Networking, Springer, 1996, pp.~493--498.

\bibitem{SaaS86}
{\sc Y.~Saad and M.~H. Schultz}, {\em {GMRES}: a generalized minimal residual
  algorithm for solving nonsymmetric linear systems}, SIAM Journal on
  Scientific and Statistical Computing, 7 (1986), pp.~856--869.

\bibitem{sctu:2016}
{\sc J.~A. Scott and M.~T\accent23uma}, {\em Preconditioning of linear least
  squares by robust incomplete factorization for implicitly held normal
  equations}, SIAM Journal on Scientific Computing, 38 (2016), pp.~C603--C623.

\bibitem{sctu:2017b}
{\sc J.~A. Scott and M.~T\accent23uma}, {\em Solving mixed sparse--dense linear
  least-squares problems by preconditioned iterative methods}, SIAM Journal on
  Scientific Computing, 39 (2017), pp.~A2422--A2437.

\bibitem{ScoT21}
{\sc J.~A. Scott and M.~T\accent23uma}, {\em Strengths and limitations of
  stretching for least-squares problems with some dense rows}, ACM Transactions
  on Mathematical Software, 41 (2021), pp.~1:1---1:25.

\bibitem{SmiBG96}
{\sc B.~F. Smith, P.~E. Bj{\o}rstad, and W.~D. Gropp}, {\em Domain
  decomposition: parallel multilevel methods for elliptic partial differential
  equations}, Cambridge University Press, 1996.

\bibitem{SpiDHNPS14}
{\sc N.~Spillane, V.~Dolean, P.~Hauret, F.~Nataf, C.~Pechstein, and
  R.~Scheichl}, {\em Abstract robust coarse spaces for systems of {PDEs} via
  generalized eigenproblems in the overlaps}, Numerische Mathematik, 126
  (2014), pp.~741--770.

\bibitem{Ste02}
{\sc G.~W. Stewart}, {\em A {K}rylov--{S}chur algorithm for large
  eigenproblems}, SIAM Journal on Matrix Analysis and Applications, 23 (2002),
  pp.~601--614.

\bibitem{TanNVE09}
{\sc J.~M. Tang, R.~Nabben, C.~Vuik, and Y.~A. Erlangga}, {\em Comparison of
  two-level preconditioners derived from deflation, domain decomposition and
  multigrid methods}, Journal of Scientific Computing, 39 (2009), pp.~340--370.

\bibitem{VanSG09}
{\sc J.~Van~lent, R.~Scheichl, and I.~G. Graham}, {\em Energy-minimizing coarse
  spaces for two-level {Schwarz} methods for multiscale {PDEs}}, Numerical
  Linear Algebra with Applications, 16 (2009), pp.~775--799.

\end{thebibliography}
